\documentclass[reqno]{amsart}
\usepackage[scale=0.75, centering, headheight=14pt]{geometry}
\usepackage[latin1]{inputenc}
\usepackage[T1]{fontenc}
\usepackage{lmodern}
\usepackage[english]{babel}

\usepackage{stmaryrd}
\usepackage{tikz}
\usepackage{bbm}
\usepackage{amsmath,amssymb,amsfonts,amsthm}
\usepackage{mathtools,accents}
\usepackage{mathrsfs}
\usepackage{xfrac}
\usepackage{array} 
\usepackage{aliascnt}
\usepackage{booktabs} % for much better looking tables
\usepackage{array} % for better arrays (eg matrices) in maths

\usepackage{verbatim} % adds environment for commenting out blocks of text & for better verbatim
\usepackage{subfig} % make it possible to include more than one captioned figure/table in a single float
\usepackage{mathrsfs}
\usepackage{mathrsfs, dsfont}
\usepackage{amssymb}
\usepackage{amsthm}
\usepackage{amsmath,amsfonts,amssymb,esint}
\usepackage{graphics,color}
\usepackage{enumerate}
\usepackage{mathtools,centernot}

\usepackage{microtype}
\usepackage{paralist} % very flexible & customisable lists (eg. enumerate/itemize, etc.)
\usepackage{cases}
\usepackage[initials]{amsrefs}
\allowdisplaybreaks

\usepackage{braket}
\usepackage{bm}

\usepackage[citecolor=blue,colorlinks]{hyperref}
\addto\extrasenglish{}

\usepackage{enumerate}
\usepackage{xcolor}

\usepackage{aliascnt}

\makeatletter
\def\newaliasedtheorem#1[#2]#3{
  \newaliascnt{#1@alt}{#2}
  \newtheorem{#1}[#1@alt]{#3}
  \expandafter\newcommand\csname #1@altname\endcsname{#3}
}
\makeatother

\theoremstyle{plain}
\newtheorem{theorem}{Theorem}[section]
\newaliasedtheorem{lemma}[theorem]{Lemma}
\newaliasedtheorem{prop}[theorem]{Proposition}
\newaliasedtheorem{claim}[theorem]{Claim}
\newaliasedtheorem{corollary}[theorem]{Corollary}
\newtheorem*{NTEO}{Theorem}

\newtheorem*{NCOR}{Corollary}
\theoremstyle{remark}

\theoremstyle{definition}
\newaliasedtheorem{definition}[theorem]{Definition}
\newaliasedtheorem{example}[theorem]{Example}
\newtheorem{Cond}{Condition}

\theoremstyle{remark}
\newaliasedtheorem{remark}[theorem]{Remark}
\newaliasedtheorem{ex}[theorem]{Example}

\numberwithin{equation}{section}

\def\eps{\varepsilon}
\DeclareMathOperator{\tr}{tr}
\def\R{\mathbb R}
\def\N{{\mathbb N}}% nonnegative integers
\DeclareMathOperator{\dv}{div}
\DeclareMathOperator{\Lip}{Lip}
\DeclareMathOperator{\spn}{span}
\DeclareMathOperator{\curl}{curl}

\DeclareMathOperator{\Ker}{Ker}

\DeclareMathOperator{\rank}{rank}

\DeclareMathOperator{\Sym}{Sym}
\DeclareMathOperator{\cof}{cof}
\DeclareMathOperator{\id}{id}
\newcommand{\M}{\mathbb{M}}
\DeclareMathOperator{\im}{Im}

\newcommand{\E}{\mathds{E}}
\DeclareMathOperator{\spt}{spt}

\DeclareMathOperator{\dist}{d}
\title{On the constancy theorem for anisotropic energies through differential inclusions}

\author[J. Hirsch and  R. Tione]{Jonas Hirsch \and Riccardo Tione}
\address{Jonas Hirsch
\hfill\break Universit\"at Leipzig, Mathematisches Institut, Augustusplatz 10, 04109 Leipzig, Germany}
\email{hirsch.jonas@math.uni-leipzig.de}
\address{Riccardo Tione  
\hfill\break  EPFL B, Station 8, CH-1015 Lausanne, CH}
\email{riccardo.tione@epfl.ch}

\begin{document}
\maketitle

\begin{abstract}
In this paper we study stationary graphs for functionals of geometric nature defined on currents or varifolds. The point of view we adopt is the one of differential inclusions, introduced in this context in the recent paper \cite{DLDPKT}. In particular, given a polyconvex integrand $f$, we define a set of matrices $C_f$ that allows us to rewrite the stationarity condition for a graph with multiplicity as a differential inclusion. Then we prove that if $f$ is assumed to be non-negative, then in $C_f$ there is no $T'_N$ configuration, thus recovering the main result of \cite{DLDPKT} as a corollary. Finally, we show that if the hypothesis of non-negativity is dropped, one can not only find $T'_N$ configurations in $C_f$, but it is also possible to construct via convex integration a very degenerate stationary point with multiplicity.
\end{abstract}

\tableofcontents

\section{Introduction}

In this paper we continue the study started in \cite{DLDPKT,TR} of functionals arising from geometric variational problems from the point of view of differential inclusions. The energies we consider are of the form
\begin{equation}\label{EN}
\Sigma_\Psi(T)\doteq \int_{E}\Psi(\vec T(x))\theta(x)d\mathcal{H}^m(x),
\end{equation}
defined on $m$-dimensional rectifiable currents (resp. varifolds) $T = \llbracket E,\vec T, \theta \rrbracket$ of $\Omega\times \R^n$, where $\Omega \subset \R^m$ is a convex and bounded open set, and the integrand $\Psi$ is defined on the oriented (resp. non-oriented) Grassmanian space. In order to keep the technicalities at a minimum level, we defer all the definitions of these geometric objects to Section \ref{geom}. The main interest is the regularity of stationary points for energies as in \eqref{EN} satisfying suitable \emph{ellipticity} conditions. From the celebrated regularity theorem of Allard of \cite{ALLARD}, it is known that an \emph{$\eps$-regularity theorem} holds for stationary points of the area functional, namely the case in which $\Psi \equiv 1$. Since then, the question of extending this result to more general energies has remained open. On the other hand, the situation is more understood for minimizers of energies of the form \eqref{EN}, where similar partial regularity theorems are known, see for instance \cite[Ch. 5]{FED}, \cite{SchoenSimon}.
\\
\\
In \cite{DLDPKT}, the second author togheter with C. De Lellis, G. De Philippis and B. Kirchheim already approached this regularity problem through the viewpoint of differential inclusions. Since this work is also based on that viewpoint, let us briefly explain what this means. The strategy of \cite{DLDPKT} consisted first in rewriting \eqref{EN} on a special class of geometric objects, namely multiplicity one graphs of Lipschitz maps, and study the differential inclusion associated to the system of PDEs arising from the stationarity condition. Namely, it can be shown that, see \cite[Sec. 6]{DLDPKT} or Subsection \ref{GM}, to a $C^k$ integrand $\Psi$  as the one appearing in \eqref{EN}, one can naturally associate a $C^k$ function $f: \R^{n\times m}\to\R$ with the property that
\begin{equation}\label{ENF}
\E_f(u) \doteq \int_{\Omega}f(Du(x))dx= \Sigma_{\Psi}(T_u).
\end{equation}
%\[
%\Sigma_\Psi({\color{blue}$\mathbf{G}_u}$}) = \E_f(u),
%\]
where $T_u = \llbracket\Gamma_u,\vec\xi_u,1\rrbracket$ is the current associated to the graph of $u$ i.e. if $v(x)\doteq(x,u(x))$ is the graph map we have $T_u={v}_{\#} \llbracket \Omega \rrbracket$.
%$\Gamma_u = \{(x,u(x)): x \in \Omega\}$, i.e. as soon as $T$ represents integration over the graph of a Lipschitz function $u$ with multiplicity one, and
%\begin{equation}\label{ENF}
%\E_f(u) = \int_{\Omega}f(Du(x))dx.
%\end{equation}
In particular, it is possible to prove, see \cite[Prop. 6.8]{DLDPKT} that $T_u$ is stationary for the energy \eqref{EN} if and only if $u$ solves the following equations:
\begin{equation}\label{outer}
\int_{\Omega}\langle Df(D u),D v\rangle dx = 0, \quad\forall v\in C^1_c(\Omega,\mathbb{R}^n) \vspace{1mm}
\end{equation}
and
\begin{equation}\label{inner}
\int_{\Omega}\langle Df(D u), D u D \phi\rangle dx - \int_{\Omega}f(D u)\dv \phi \; dx = 0,\quad  \forall \phi\in C_c^1(\Omega,\mathbb{R}^m).
\end{equation}
The Euler-Lagrange equation \eqref{outer} corresponds to variations of the form
\[
\frac{d}{d\eps}|_{\eps = 0}\E_f(u + \eps v) = 0,
\]
usually called \emph{outer} variations, and \eqref{inner} corresponds to variations of the form
\[
\frac{d}{d\eps}|_{\eps = 0}\E_f(u\circ(x + \eps \Phi)) = 0,
\]
called \emph{inner} (or \emph{domain}) variations. The second step is to study \eqref{outer} and \eqref{inner} from the point of view of differential inclusions. This amounts to rewrite \eqref{outer}-\eqref{inner} equivalently as
\begin{equation}\label{diffinc}
\left(
\begin{array}{c}
Du\\
A\\
B
\end{array}
\right) \in K_f\doteq
\left\{C\in \R^{(2n + m)\times m}: C =
\left(
\begin{array}{cc}
X\\ 
Df(X)\\
X^TDf(X) - f(X)\id
\end{array}
\right)\right\},
\end{equation}
for $A \in L^\infty(\Omega,\R^{n\times m})$, $B \in L^\infty(\Omega,\R^m)$ with $\dv(A) = 0$, $\dv(B) = 0$.
\\
\\
This paper focuses on the same problem as  \cite{DLDPKT}, i.e. regularity of stationary points for geometric integrands, but with the addition of considering graphs with arbitrary positive multiplicity. This of course enlarges the class of competitors and might allow for more flexibility in the regularity of solutions. In particular, we consider \emph{polyconvex} functions $f$, i.e. $$f(X) = g(X,\Phi(X)),$$ where $g \in C^1(\R^k)$ is a convex function and $\Phi:\R^{n\times m} \to \R^k$ is the vector containing all the minors (subdeterminants) of order larger than or equal to 2 of $X \in \R^{n\times m}$. In analogy with \eqref{outer}-\eqref{inner}, we will be interested in the following system of PDEs
\begin{equation}\label{PDEBETA}
\left\{
\begin{array}{ll}
\displaystyle \int_{\Omega}\langle Df(D u),D v\rangle\beta dx = 0 &\forall v\in C^1_c(\Omega,\mathbb{R}^n) \vspace{1mm}\\ 
\displaystyle \int_{\Omega}\langle Df(D u), D u D \phi\rangle \beta \;dx - \int_{\Omega}f(D u)\dv \phi \beta \;dx = 0\qquad & \forall \phi\in C_c^1(\Omega,\mathbb{R}^m).
\end{array}\right.
\end{equation}
for a Lipschitz map $u \in \Lip(\Omega,\R^n)$, and a Borel function $\beta \in L^\infty(\Omega,\R^+)$. The study of objects with multiplicity is rather natural in the context of stationary rectifiable varifolds or currents. When dealing with these objects, one is interested in showing a so-called \emph{constancy theorem}, see \cite[Theorem 8.4.1]{SIM}. A constancy theorem in the sense of \cite[Theorem 8.4.1]{SIM} asserts that if a stationary (for the area) varifold of dimension $m$ has support contained in a $C^2$ manifold of the same dimension, then the varifold must be given by a fixed multiple of the manifold, so that in particular the multiplicity must be constant. In \cite{DUG}, it was shown that instead of $C^2$, even Lipschitz regularity of the manifold is sufficient to guarantee the validity of the Constancy Theorem. This is connected to the following algebraic fact. If a $C^2$ map $u$ solves \eqref{outer}, then it necessarily solves also \eqref{inner}, hence the system \eqref{outer}-\eqref{inner} reduces to equation \eqref{outer}. Nonetheless, if $u \in C^2$ and solves \eqref{PDEBETA} for a bounded multiplicity $\beta$, then it is not anymore true that $u$ automatically solves the first. One therefore would like to show a priori that the multiplicity is constant and subsequently one is again in the situation given by \eqref{outer}-\eqref{inner}. As for regularity theorems, no general constancy result is known at the moment for general functionals, except for the codimension one case, see \cite{DDH}.
\\
\\
As said, the tools we use are the same as the ones of \cite{DLDPKT}, namely we rewrite \eqref{PDEBETA} as
\begin{equation}\label{diffincf}
\left(
\begin{array}{c}
Du\\
A\\
B
\end{array}
\right) \in
C_f\doteq
\left\{C\in \R^{(2n + m)\times m}: C =
\left(
\begin{array}{cc}
X\\ 
\beta Df(X)\\
\beta  X^TDf(X) - \beta f(X)\id
\end{array}
\right), \text{ for some $\beta$ > 0}\right\},
\end{equation}
again for $A \in L^\infty(\Omega,\R^{n\times m})$, $B \in L^\infty(\Omega,\R^m)$ with $\dv(A) = 0$, $\dv(B) = 0$. Our result is twofold. First, we will show that, if $f$ is assumed to be non-negative, then the same result as \cite[Theorem 1]{DLDPKT} holds, namely in $C_f$ there are no $T'_N$ configurations. Secondly, we show the optimality of this result by proving that if we drop the hypothesis on the positivity of $f$, one can not only embed a special family of matrices in $C_f$, but one can actually construct a stationary current for the energy give in \eqref{EN} whose support lies on the graph of a Lipschitz and nowhere $C^1$ map. In order to formulate properly these results, we need some terminology concerning differential inclusions.
\\
\\
Differential inclusions are relations of the form 
\begin{equation}\label{difgen}
M(x) \in K \subset \R^{n\times m} \text{ a.e. in }\Omega
\end{equation}
for $M \in L^\infty(\Omega,\R^{n\times m})$ satisfying $\mathscr{A}(M) = 0$ in the weak sense for some constant coefficients, linear differential operator $\mathscr{A}(\cdot)$. To every operator $\mathscr{A}(\cdot)$, one can associate a \emph{wave cone}, denoted with $\Lambda_\mathscr{A}$, that is made of those directions $A$ in which it is possible to have \text{plane wave} solution, i.e. $A \in {\Lambda_{\mathscr{A}}}$ if and only if there exists $\xi \in \R^m$ such that
\[
\mathscr{A}(h((x,\xi))A) = 0,\quad \forall h \in C^1(\R).
\]
In this work, we will not need to consider various differential operators, as we will only work with the mixed div-curl operator introduced in \eqref{diffinc}. In that case, we denote the cone with $\Lambda_{dc}$ and we will introduce it in detail in Section \ref{DI}. Due to the connection of the wave-cone to the existence of oscillatory solution of \eqref{difgen}, a very first step to exclude \text{wild} solutions of \eqref{difgen} is to check that
\begin{equation}\label{RANK?}
A - B \notin \Lambda_{\mathscr{A}}, \quad \forall A,B \in K.
\end{equation}
This is usually quite simple to verify, and indeed we will show in Proposition \ref{RANK} that, if $f$ is positive, then \eqref{RANK?} holds with $\Lambda_{\mathscr{A}} = \Lambda_{dc}$ and $K$ replaced by $C_f$. Property \eqref{RANK?} is in general not sufficient to guarantee good regularity properties of solution of \eqref{difgen}. Indeed, in \cite{SMVS}, S. M\"uller and V. \v Sver\'ak constructed a striking counterexample to elliptic regularity for solutions of
\begin{equation}\label{diffincms}
Dv(x) \in K'_f \doteq \left\{C\in \R^{4\times 2}: C =
\left(
\begin{array}{cc}
X\\ 
Df(X)J
\end{array}
\right)\right\}\subset \R^{4\times 2},
\end{equation}
where the function $f \in C^\infty(\R^{2\times 2})$ is quasiconvex (for the definition of quasiconvex function, we refer the reader to \cite{SMVS}), and $J$ is a matrix satisfying $J = -J^T$ and $J^2 = -\id$. In particular, they were able to show that there exists a Lipschitz and nowhere $C^1$ function $v: \Omega \subset \R^2 \to \R^4$ satisfying the differential inclusion \eqref{diffincms}. Their strategy was subsequently improved by L. Sz{\'{e}}kelyhidi in \cite{LSP} showing that $f$ can be chosen polyconvex. In both cases, $K'_f$ does not contain rank one connections, i.e.
\[
\rank(A-B) = 2, \quad \forall A,B \in K'_f,
\]
and this can proved to be equivalent to \eqref{RANK?} in the case $\mathscr{A} = \curl$. Their strategy was based on showing that in $K'_f$ other suitable families of matrices could be embedded, the so-called $T_N$ configurations. In our situation, since we are dealing with mixed div-curl operators, we need to consider a slightly different version of $T_N$ configurations, that we have named $T'_N$ configurations in \cite{DLDPKT}. We postpone the definition of $T_N$ and $T'_N$ configurations to Section \ref{POSCASE}, but we are finally able to formally state our main positive result:

\begin{NTEO}
If $f\in C^1 (\R^{n\times m})$ is a strictly polyconvex function, then $C_f$ does not contain any set $\{A_1, \ldots , A_N\} \subset \R^{(2n + m)\times m}$ which induces a $T_N'$ configuration, provided that $f(X_1) \ge 0,\dots, f(X_N) \ge 0$, if
\[
A_i = \left(
\begin{array}{c}
X_i\\
Y_i\\
Z_i
\end{array}
\right), \quad X_i,Y_i \in \R^{n\times m}, Z_i \in \R^{m\times m}, \forall i \in \{1,\dots, N\}.
\]
\end{NTEO}

This result, as \cite[Theorem 1]{DLDPKT}, shows that it is not possible to apply the convex integration methods of \cite{SMVS,LSP} to show the existence of an irregular solution of the system \eqref{PDEBETA}. This theorem is stronger than \cite[Theorem 1]{DLDPKT}, in the sense that we are able to show \cite[Theorem 1]{DLDPKT} as a corollary:

\begin{NCOR}
If $f\in C^1 (\R^{n\times m} )$ is a strictly polyconvex function (not necessarily non-negative), then $K_f$ does not contain any set $\{A_1, \ldots , A_N\}$ which induces a $T_N'$ configuration.
\end{NCOR}

Finally, in Section \ref{SCC}, we show the optimality of the hypothesis of non-negativity of the previous theorem by proving the following:

\begin{NTEO}
There exists a smooth and elliptic integrand $\Psi: \Lambda_2(\R^4)\to \R$ such that the associated energy $\Sigma$ admits a stationary point $T$ whose (integer) multiplicities are not constant. Moreover the rectifiable set supporting $T$ is given by a graph of a Lipschitz map $u: \Omega \to \R^2$ that fails to be $C^1$ in any open subset $\mathcal{V} \subset \Omega$.
\end{NTEO}

The last Theorem is obtained by embedding in the differential inclusion \eqref{diffincf} what has been named in \cite{FS} \emph{large} $T_N$ configuration. Following the strategy of \cite{LSP}, we do not a priori choose a polyconvex $f \in C^\infty(\R^{2\times 2})$, but rather we construct it in such a way that $C_f$ already contains this special family of matrices. Once the polyconvex function $f$ has been built, we prove an extension result for $f$ to the Grassmanians, thus obtaining the integrand $\Psi$ of the statement of the Theorem. The extension results are quite simple and might be of independent interest. The construction of our counterexample can not be carried out in the varifold setting. The reason is quite elementary, as the integrand $\Psi$ we would need to construct in the varifold case should be even, convex and positively $1$-homogeneous, hence positive. We refer the reader to Remark \ref{posvarint} for more details. Moreover, let us point out that positivity of the integrand is a necessary assumption when studying existence of minima, but to the best of our knowledge there is no available example for it to be a necessary assumption also when studying regularity properties of stationary points.
\\
\\
The paper is organized as follows. In Section \ref{POSCASE}, we recall the statements of our main results in the case of non-negative integrands $f$ and we collect some crucial preliminary results of \cite{DLDPKT}. The proof of the main results in the positive case, i.e. Proposition \ref{RANK}, Theorem \ref{t:pos} and Corollary \ref{t:main}, will be given in Section \ref{MT}. In Section \ref{SCC}, we provide a counterexample to regularity when dropping the hypothesis of positivity of the integrand. Some lemmas of Section \ref{SCC} concerning the extension of polyconvex functions to the Grassmaniann manifold can be easily extended to general dimension and codimensions. Therefore, we give the proof of these general versions in Section \ref{genext}. Finally, the appendix contains a concise introduction to the tools of geometric measure theory used along the paper.

\subsection*{Aknowledgements}

The authors would like to thank Camillo De Lellis for his interest in the problem and some preliminary discussions. This work was developed while R. T. was finishing his PhD at the University of Z\"urich, and is now supported by the SNF Grant 200021$\_$182565. J. H. was partially supported by the German Science Foundation DFG in the context of the Priority Program SPP 2026 \emph{Geometry at Infinity}.

\section{Positive case: absence of $T_N$ configurations}\label{POSCASE}

In this section we collect some preliminary results proved in \cite{DLDPKT}, that will be essential for the proofs of the next section.

\subsection{Div-curl differential inclusions, wave cones and inclusion sets}\label{DI}

In this subsection, we explain how to rephrase the system \eqref{outer}-\eqref{inner} as a differential inclusion. As recalled in the introduction, the Euler-Lagrange equations defining stationary points for energies $\mathds{E}_f$ are the couple of equations \eqref{outer}, \eqref{inner}, that can be written in the classical form:
\[
\begin{cases}
\dv(Df(Du)) = 0  \\ 
\dv(Du^TDf(Du) - f(Du)\id) = 0  
\end{cases}
\]
Thus we are lead to study the following
%
%
%
%of the form
%\begin{equation}\label{e:energy}
%\mathds{E}_f(u) = \int_{\Omega}f(Du(x))dx
%\end{equation}
%are given by:
%\begin{equation}\label{vargr}
%\left\{
%\begin{array}{ll}
%\displaystyle \int_{\Omega}\langle Df(D u),D v\rangle dx = 0 &\forall v\in C^1_c(\Omega,\mathbb{R}^n) \vspace{1mm}\\ 
%\displaystyle \int_{\Omega}\langle Df(D u), D u D \phi\rangle dx - \int_{\Omega}f(D u)\dv \phi \;dx = 0\qquad & \forall \phi\in C_c^1(\Omega,\mathbb{R}^m),
%\end{array}\right.
%\end{equation}
%Here we rewrite the system \eqref{vargr} as a differential inclusion. To do so, it is sufficient to notice that the left hand side of the second equation can be rewritten as
%\begin{align*}
%\int_{\Omega}\langle Df(D u), D u D \phi\rangle dx - \int_{\Omega}f(D u)\dv \phi dx
%& =\int_{\Omega}\langle D u^T Df(D u), D\phi \rangle - \langle f(D u)\id,D g\rangle dx\\
%& =\int_{\Omega}\langle D u^T Df(D u) - f(D u)\id, D \phi\rangle dx
%\end{align*}
%Hence, the inner variation equation is the weak formulation of
%\[
%\dv( D u^T Df(D u) - f(D u)\id) = 0.
%\]
%Since also the outer variation is the weak formulation of a PDE in divergence form, namely
%\[
%\dv( Df(D u)) = 0,
%\]
%we consider the following 
{\em div-curl differential inclusion} for a triple of maps $X, Y\in L^\infty (\Omega, \mathbb{R}^{n\times m})$ and $Z\in L^\infty (\Omega, \R^{m\times m})$:
\begin{equation}\label{e:div_curl_free}
{\curl}\, X = 0, \qquad {\dv}\, Y =0, \qquad {\dv}\, Z = 0\, ,
\end{equation}
\begin{equation}\label{e:inclusion}
W \doteq  \left( 
\begin{array}{c}
X\\
Y\\
Z
\end{array}
\right) 
\in K_f =  \left\{A \in \R^{(2n + m)\times m}:
A =
\left(
\begin{array}{c}
X\\
Df(X)\\
X^TDf(X) - f(X)\id
\end{array}
\right)
\right\},
\end{equation}
where $f\in C^1 (\R^{n\times m})$ is a fixed function.
\\
\\
Moreover, we also consider the following more general system of PDEs, for $u \in \Lip(\Omega, \R^n)$ and a Borel map $\beta \in L^\infty(\Omega,(0,+\infty))$:
\begin{equation}\label{vargrb}
\left\{
\begin{array}{ll}
\displaystyle \int_{\Omega}\langle Df(D u),D v\rangle\beta dx = 0 &\forall v\in C^1_c(\Omega,\mathbb{R}^n) \vspace{1mm}\\ 
\displaystyle \int_{\Omega}\langle Df(D u), D u D \phi\rangle \beta dx - \int_{\Omega}f(D u)\dv \phi \beta \;dx = 0\qquad & \forall \phi\in C_c^1(\Omega,\mathbb{R}^m).
\end{array}\right.
\end{equation}
\noindent
This system is equivalent to the stationarity in the sense of varifolds of the varifold $V = \llbracket\Gamma_u,\beta\rrbracket$, where $\Gamma_u$ is the graph of $u$. This is discussed in Subsection \ref{GM}. The div-curl differential inclusion associated to this system is,  again for a triple of maps $X, Y\in L^\infty (\Omega, \mathbb{R}^{n\times m})$ and $Z\in L^\infty (\Omega, \R^{m\times m})$:
\begin{equation}\label{e:div_curl_freeb}
{\curl}\, X = 0, \qquad {\dv}\, Y =0, \qquad {\dv}\, Z = 0\, ,
\end{equation}
\begin{equation}\label{e:inclusionb}
W \doteq  \left( 
\begin{array}{c}
X\\
Y\\
Z
\end{array}
\right) 
\in C_f
\end{equation}
where
\begin{equation}\label{CF}
C_f =
\left\{C\in \R^{(2n + m)\times m}: C =
\left(
\begin{array}{cc}
X\\ 
\beta Df(X)\\
\beta  X^TDf(X) - \beta f(X)\id
\end{array}
\right), \text{ for some $\beta$ > 0}\right\},
\end{equation}
This discussion proves the following

\begin{lemma}\label{equiv}
Let $f\in C^1 (\R^{n\times m})$. A map
$u \in \Lip(\Omega,\mathbb{R}^n)$ is a stationary point of the energy \eqref{ENF} if and only there are matrix fields $Y\in L^\infty (\Omega, \R^{n\times m})$ and $Z\in L^\infty (\Omega, \R^{m\times m})$ such that $W = (Du, Y,Z)$ solves the div-curl differential inclusion \eqref{e:div_curl_free}-\eqref{e:inclusion}.
\\
\indent Moreover, the couple $(u,\beta) \in \Lip(\Omega,\mathbb{R}^n)\times L^\infty(\Omega,(0,+\infty))$ solves \eqref{vargrb} if and only there are matrix fields $Y\in L^\infty (\Omega, \R^{n\times m})$ and $Z\in L^\infty (\Omega, \R^{m\times m})$ such that $W = (Du, Y,Z)$ solves the div-curl differential inclusion \eqref{e:div_curl_freeb}-\eqref{e:inclusionb}.
\end{lemma}

Finally, we introduce here the wave-cone associated to the mixed \emph{div-curl} operator that is relevant for us.

\begin{definition}\label{d:cone_dc} The cone $\Lambda_{dc}\subset \mathbb R^{(2n+m)\times m}$ consists of the matrices in block form
\[
\left(
\begin{array}{l}
X\\
Y\\
Z
\end{array}\right)
\]
with the property that there is a direction $\xi\in \mathbb S^{m-1}$ and a vector $u\in \mathbb R^n$ such that $X = u\otimes \xi$, $Y \xi =0$ and $Z\xi =0$. 
\end{definition}

\subsection{$T_N$ configurations and $T'_N$ configurations} We start defining $T_N$ configurations for \emph{classical} curl-type differential inclusions.

\begin{definition}\label{definition_TN} An ordered set of $N\geq 2$ matrices $\{X_i\}_{i=1}^N \subset \R^{n\times m}$ of distinct matrices is said to \emph{induce a $T_N$ configuration} if there exist matrices $P, C_i \in\R^{n\times m}$ and real numbers $k_i > 1$ such that:
\begin{itemize}
\item[(a)] Each $C_i$ belongs to the wave cone of ${\curl}\, X=0$, namely $\rank (C_i) \leq 1$ for each $i$;
\item[(b)] $\sum_i C_i = 0$;
\item[(c)] $X_1, \ldots, X_N$, $P$ and $C_1, \ldots, C_N$ satisfy the following $N$ linear conditions 
\begin{equation}\label{INCform}
\begin{split}
&X_1 = P + k_1 C_1 ,\\
&X_2 = P + C_1 + k_2C_2 ,\\
&\dots\\
&\dots\\
&X_N = P + C_1 +\dots + k_NC_N\, .
\end{split}
\end{equation}
\end{itemize}
In the rest of the chapter we will use the word $T_N$ configuration for the data $$P, C_1, \ldots , C_N, k_1, \ldots k_N.$$ We will moreover say that the configuration is {\em nondegenerate} if $\rank (C_i)=1$ for every $i$. 
\end{definition}

As in \cite{DLDPKT}, we give a slightly more general definition of $T_N$ configuration than the one usually given in the literature (cf. \cite{SMVS,LSP,LSR}), in that we drop the requirement that there are no rank-one connections between distinct $X_i$ and $X_j$. We refer the reader to \cite{DLDPKT} for discussions concerning $T_N$ configurations.
\\
\\
Adapted to the div-curl operator we introduce $T'_N$ configurations, originally introduced in \cite{DLDPKT}.%For the mixed div-curl operator we are interested in, in \cite{DLDPKT} the authors introduced the so called $T'_N$ configurations. In particular, the geometric idea behind the construction of $T_N$ and $T'_N$ configurations is the same, but the wave cone condition is replaced by the one dictated by the new system of PDEs \eqref{e:div_curl_free}.

\begin{definition}\label{definition_TN'}
A family $\{A_1, \ldots, A_N\}\subset \R^{(2n+m)\times m}$ of $N\geq 2$ {\em distinct}
\[
A_i\doteq \left(
\begin{array}{c}
X_i\\
Y_i\\
Z_i
\end{array}
\right)
\]
induces a {\em $T_N'$ configuration} if there are matrices $P, Q,  C_i, D_i \in \R^{n\times m}$, $R, E_i\in \R^{m\times m}$ and coefficients $k_i >1$ such that
\begin{equation}\label{formTprime}
\left(
\begin{array}{c}
X_i\\
Y_i\\
Z_i
\end{array}
\right) 
= \left(
\begin{array}{c}
P\\
Q\\
R
\end{array}
\right)
+
 \left(
\begin{array}{c}
C_1\\
D_1\\
E_1
\end{array}
\right)
+ \cdots 
+
\left(
\begin{array}{c}
C_{i-1}\\
D_{i-1}\\
E_{i-1}
\end{array}
\right)
+
k_i
\left(
\begin{array}{c}
C_i\\
D_i\\
E_i
\end{array}
\right)
\end{equation}
and the following properties hold:
\begin{itemize}
\item[(a)] each element $(C_i, D_i, E_i)$ belongs to the wave cone $\Lambda_{dc}$ of \eqref{e:div_curl_free}; 
\item[(b)] $\sum_\ell C_\ell = 0$, $\sum_\ell D_\ell =0 $ and $\sum_\ell E_\ell = 0$.
\end{itemize}
We say that the $T'_N$ configuration is {\em nondegenerate} if $\rank (C_i)=1$ for every $i$.
\end{definition}
\noindent
We collect here some simple consequences of the definition above.

\begin{prop}\label{p:T_N'_easy}
Assume $A_1, \ldots , A_N$ induce a $T_N'$ configuration with $P,Q, R, C_i, D_i, E_i$ and $k_i$ as in Definition \ref{definition_TN'}. Then:
\begin{itemize}
\item[(i)] $\{X_1, \ldots , X_N\}$ induce a $T_N$ configuration of the form \eqref{INCform}, if they are distinct; moreover the $T_N'$ configuration is nondegenerate if and only if the $T_N$ configuration induced by $\{X_1, \ldots , X_N\}$ is nondegenerate;
\item[(ii)] For each $i$ there is an $n_i\in \mathbb S^{m-1}$ and a $u_i\in \R^n$ such that
$C_i = u_i\otimes n_i$, $D_i n_i =0$ and $E_i n_i =0$;
\item[(iii)] $\tr C_i^T D_i = \langle C_i, D_i\rangle = 0$ for every $i$.
\end{itemize}
\end{prop}

\subsection{Strategy}\label{STRA}

Before starting with the proof of the main result of this chapter, it is convenient to explain the strategy we intend to follow. In order to do so, let us consider here the case $n = m = 2$, $N = 5$. Suppose by contradiction that there exists a strictly polyconvex function $f : \R^{2\times 2 }\to \R$, $f(X) = g(X,\det(X))$ and a $T'_5$ configuration $A_1,A_2,A_3,A_4,A_5$,
\[
A_i =
\left(
\begin{array}{c}
X_i\\
Y_i\\
Z_i
\end{array}
\right),\quad \forall i \in \{1,\dots,5\},
\]
where $X_i,Y_i,Z_i$ fulfill the relations of \eqref{formTprime}, i.e. 
\[
\left(
\begin{array}{c}
X_i\\
Y_i\\
Z_i
\end{array}
\right) 
= \left(
\begin{array}{c}
P\\
Q\\
R
\end{array}
\right)
+
 \left(
\begin{array}{c}
C_1\\
D_1\\
E_1
\end{array}
\right)
+ \cdots 
+
\left(
\begin{array}{c}
C_{i-1}\\
D_{i-1}\\
E_{i-1}
\end{array}
\right)
+
k_i
\left(
\begin{array}{c}
C_i\\
D_i\\
E_i
\end{array}
\right).
\]
We will see below that we can without loss of generality assume that $P=0$. The first part of the strategy follows the same lines of the one of \cite{DLDPKT}. Indeed, we think the relations $A_i \in C_f$, $\forall i$, where $C_f$ has been defined in \eqref{CF}, as two separate pieces of information:
\begin{equation}\label{info1}
\left(
\begin{array}{c}
X_i\\
Y_i
\end{array}
\right) \in K'_f = \left\{A \in \R^{4\times 2}: A = \left(
\begin{array}{c}
X\\
\beta Df(X)
\end{array}
\right), \beta >0, X \in \R^{2\times 2}\right\}
\end{equation}
and
\begin{equation}\label{info2}
Z_i = X_i^TY_i - \beta_if(X_i)\id.
\end{equation}
Let us denote with $c_i\doteq f(X_i)$. As in \cite{DLDPKT}, we use \eqref{info1} to obtain inequalities involving $X_i,Y_i$ and quantities involving $f$. These are deduced from the polyconvexity of $f$, analogously to \cite[Lemma 3]{LSP}. In particular, \eqref{info1} is rewritten as
\begin{equation}\label{info3}
c_i - c_j +\frac{1}{\beta_i}\langle Y_i,X_j - X_i\rangle  - d_i\det(X_i - X_j) < 0,
\end{equation}
for $d_i \doteq \partial_{y_5}g(y_1,y_2,y_3,y_4,y_5)|_{(X_i,\det(X_i))}$. This is proved in Proposition \ref{p:convexity}. The final goal is to prove that these inequalities can not be fulfilled at the same time. Then, as in \cite{DLDPKT}, we can simplify \eqref{info3} using the structure result on $T_N$ configurations in $\R^{2\times 2}$ of \cite[Proposition 1]{LSR}. This asserts, in the specific case of the ongoing example, the existence of 5 vectors $(t_1^i,\dots, t_5^i), i \in \{1,\dots, 5\}$ with positive components, such that
\begin{equation}\label{tt}
\sum_{j = 1}^5t_j^i\det(X_j - X_i) = 0,\quad \forall i \in\{1,\dots, 5\}.
\end{equation}
If we use this result in \eqref{info3}, we can eliminate from the expression the variable $d_i$, thus obtaining
\begin{align*}
\nu_i &\doteq \sum_{j = 1}^5t_j^i(c_i - c_j +\frac{1}{\beta_i}\langle Y_i,X_j - X_i\rangle  - d_i\det(X_i - X_j)) \\
&= \sum_{j = 1}^5t_j^i(c_i - c_j +\frac{1}{\beta_i}\langle Y_i,X_j - X_i\rangle) < 0, \quad \forall i \in \{1,\dots,5\},
\end{align*}
compare Corollary \ref{absurd}. In \cite{DLDPKT}, \cite[Proposition 1]{LSR} was extended to $T_N$ configurations in $\R^{n\times m}$, so that relations $\eqref{tt}$ remain true in every dimension and target dimension. This extension is recalled in Proposition \ref{analogm}. Despite being very useful, the last simplification can not conclude the proof. Indeed, up to now we have exploited \eqref{info1} and the fact that $\{X_1,\dots, X_5\}$ induce a $T_5$ configuration, but, if $\beta_i = 1,\forall i$, this is the exact same situation of \cite{LSP}. Since from that paper we know the existence of $T_5$ configurations in $K'_f$, clearly we can not reach a contradiction at this point of the strategy. This is where the inner variations come into play. We rewrite \eqref{info2} using the definition of $T'_5$ configuration and, after some manipulations, we find that the numbers
\[
\mu_i\doteq\sum_{j = 1}^5 t_j^i(\langle X_j - X_i ,Y_i\rangle - \beta_ic_i  + \beta_jc_j)
\]
must all be 0. For the index $I$ such that $\beta_I = \min_{i}\beta_i$, and essentially using the positivity of $c_j$, we find that
\[
0 = \mu_I = \sum_{j = 1}^5 t_j^i(\langle X_j - X_i ,Y_i\rangle - \beta_ic_i  + \beta_jc_j) \le \sum_{j = 1}^5 t_j^i(\langle X_j - X_i ,Y_i\rangle - \beta_ic_i  + \beta_I c_j) = \nu_I,
\]
which is in contradiction with the negativity of $\nu_I$.

\subsection{Preliminary results: $T_N$ configurations}\label{TNALG}

To follow the strategy explained in Subsection \ref{STRA}, we need to recall the extension of \cite[Proposition 1]{LSR} proved in \cite{DLDPKT}. Here we will only recall the essential results without proof, we refer the interested reader to \cite{DLDPKT} for the details. First, it is possible to associate to a set $T_N$-configuration of the form \eqref{INCform}, i.e.
\begin{equation*}
\begin{split}
&X_1 = P + k_1 C_1 ,\\
&X_2 = P + C_1 + k_2C_2 ,\\
&\dots\\
&\dots\\
&X_N = P + C_1 +\dots + k_NC_N\, ,
\end{split}
\end{equation*}
a \emph{defining vector} $(\lambda,\mu) \in \R^{N + 1}$, see \cite[Definition 3.7]{DLDPKT}, defined as follows:
\begin{equation}\label{defvec}
\mu\doteq \frac{k_1\dots k_{N}}{(\mu - 1)(k_1-1)\dots (k_{N} - 1)} \text{ and } \lambda_i \doteq \frac{k_1\dots k_{i - 1}}{(\mu - 1)(k_1-1)\dots (k_{i - 1} - 1)}.
\end{equation}
These relations can be inverted, in fact one can express
\begin{equation}\label{defnk}
k_i = \frac{\mu\lambda_1 + \dots + \mu\lambda_i + \lambda_{i + 1} \dots + \lambda_N}{(\mu - 1)\lambda_i}\, .
\end{equation}
Since $k_i > 1, \forall i \in \{1,\dots,N\}$, \eqref{defvec} imply that $\lambda_i > 0, \forall i, \mu > 1$ and also
\[
\sum_i \lambda_i = 1.
\]
As in \cite[Proposition 1]{LSR}, we define $N$ vectors of $\R^N$ with positive components
\begin{equation}\label{defnt}
t^{i} \doteq  \frac{1}{\xi_i}(\mu\lambda_1,\dots,\mu\lambda_{i - 1},\lambda_{i},\dots,\lambda_N), \text{ for } i \in \{1,\dots,N\},
\end{equation}
where $\xi_i > 1$ are normalization constants chosen in such a way that $\|t^{i}\|_{1} = 1$. For a vector $v = (v_1,\dots, v_N) \in \R^N$,
\[
\|v\|_1 = \sum_{j = 1}^N |v_j|.
\]
The importance of these vectors $t^i$ comes from of \cite[Proposition 1]{LSR}, where it is proved that, for a $T_N$ configuration of the form \eqref{INCform} in $\R^{2\times 2}$,
\begin{equation}\label{sum}
\sum_{j = 1}^Nt_j^iX_j = P + C_1 + \dots + C_{i - 1}
\end{equation}
Moreover, the following relation holds for every $i$:
\begin{equation}\label{sumdet1}
\det\left(\sum_{j = 1}^Nt_j^iX_j\right) = \sum_{j=1}^Nt_j^i\det(X_j)\, .
\end{equation}
We need to state the generalization of the previous relations for $T_N$ configurations of any size. In \cite[Lemma 3.10]{DLDPKT} it was proved this general Linear Algebra result:
\begin{lemma}\label{l:linear}
Assume the real numbers $\mu>1$, $\lambda_1, \ldots , \lambda_N >0$ and $k_1, \ldots , k_N >1$ are linked by the formulas \eqref{defvec}. Assume $v, v_1, \ldots, v_N, w_1, \ldots , w_N$ are elements of a vector space satisfying the  relations
\begin{align}
w_i &= v + v_1 + \ldots + v_{i-1} + k_i v_i\\
0 &= v_1+ \ldots + v_N\, .
\end{align}
If we define the vectors $t^i$ as in \eqref{defnt}, then
\begin{equation}\label{e:sumfinale}
\sum_j t^i_j w_j = v + v_1 + \ldots + v_{i-1}\, .
\end{equation}
\end{lemma}
This lemma allows to generalize \eqref{sum} and \eqref{sumdet1}, compare \cite[Proposition 3.8]{DLDPKT}. To state this result, we need some notation concerning multi- indexes. We will use $I$ for multi-indexes referring to ordered sets of rows of matrices and $J$ for multi-indexes referring to ordered sets of columns. In our specific case, where we deal with matrices in $\R^{n\times m}$ we will thus have
\begin{align*}
I &= (i_1,\dots,i_r),\qquad 1\le i_1<\dots< i_r \le n\, ,\\
 \text{ and } \qquad J &= (j_1,\dots,j_s),\qquad 1\le j_1< \dots< j_s\le m\, 
\end{align*}
and we will use the notation $|I|\doteq  r$ and $|J|\doteq s$. In the sequel we will always have $r = s$. 

\begin{definition}\label{multiind}
We denote by $\mathcal{A}_r$ the set
\[
\mathcal{A}_r = \{(I,J): |I| = |J| = r\},\qquad  1\le r \le \min(n,m) .
\]
For a matrix $M = (m_{ij})\in\R^{n\times m}$ and for $Z\in \mathcal{A}_r$ of the form $Z = (I,J)$, we denote by $M^Z$ the squared $r\times r$ matrix obtained by $M$ considering just the elements $m_{ij}$ with $i\in I$, $j\in J$ (using the order induced by $I$ and $J$). 
\end{definition}

We are finally in position to state \cite[Proposition 3.8]{DLDPKT}.

\begin{prop}\label{analogm}
Let $\{X_1, \ldots, X_N\}\subset \R^{n\times m}$ induce a $T_N$ configuration as in \eqref{INCform} with defining vector 
$(\lambda, \mu)$. Define the vectors $t^1,\dots,t^N$ as in \eqref{defnt} and for every $Z\in \mathcal{A}_r$ of order $1\le r \leq \min \{n,m\}$ define the 
minor $\mathcal{S} : \R^{n\times m} \ni X \mapsto \mathcal{S} (X) \doteq  \det (X^Z)\in \R$. Then 
\begin{equation}\label{e:sum_minor}
\sum_{j = 1}^Nt_j^i \mathcal{S}(X_j) = \mathcal{S}\left(\sum_{j = 1}^Nt_j^iX_j\right) = \mathcal{S}(P + C_1 + \dots + C_{i - 1})\, .
\end{equation}
and $A^\mu_Z \lambda = 0$.
\end{prop}

It is clear that the previous result extends \eqref{sum} and \eqref{sumdet1} to all the minors.

\subsection{Preliminary results: inclusion set associated to polyconvex functions}\label{POLYFUN}

As in \cite[Section 4]{DLDPKT}, we write a necessary condition for a set of distinct matrices $A_i \in \R^{2n\times m}$
\begin{equation}\label{inclusion}
A_i\doteq \left(
\begin{array}{c}
X_i\\
Y_i
\end{array}
\right)\, ,
\end{equation}
to belong to a set of the form
\begin{equation}\label{e:inclusion2}
K'_f \doteq \left\{\left(
\begin{array}{c}
X\\
Df(X)\\
\end{array}
\right):X \in \R^{n\times m}\right\}
\end{equation}
for some strictly polyconvex function $f:\mathbb{R}^{n\times m}\to \mathbb{R}$. First, introduce the following notation. This is the same as in \cite{DLDPKT}. Let $f:\R^{n\times m}\to \R$ be a strictly polyconvex function of the form $f(X) =g(\Phi(X))$, where $g \in C^1(\R^k)$ is strictly convex and $\Phi$ is the vector of all the subdeterminants of $X$, i.e.
\[
\Phi(X) = (X,v_1(X),\dots,v_{\min(n,m)}(X)),
\]
and $$v_s(X) = (\det(X_{Z_1}),\dots, \det(X_{Z_{\#\mathcal{A}_s}}))$$ for some fixed (but arbitrary) ordering of all the elements $Z\in\mathcal{A}_s$. Variables of $\R^k$, and hence partial derivatives in $\R^k$, are labeled using the ordering induced by $\Phi$. The first $nm$ partial derivatives, corresponding in $\Phi(X)$ to $X$, are collected in a $n\times m$ matrix denoted with $D_Xg$. The $j$-th partial derivative, $mn + 1\le j \le k$, is instead denoted by $\partial_Zg$, where $Z$ is the element of $\mathcal{A}_s$ corresponding to the $j$-th position of $\Phi$. Let us make an example in low dimension: if $n = 3,m = 2$, then $k = 9$, and we choose the ordering of $\Phi$ to be
\[
\Phi(X) = (X,\det(X_{(12,12)}),\det(X_{(13,12)}),\det(X_{(23,12)})).
\]
In this case, $y \in \R^k$ has coordinates $$y = (y_{11},y_{12},y_{21},y_{22},y_{31},y_{32},y_{(12,12)}, y_{(13,12)},y_{(23,12)}).$$ The partial derivatives with respect to the first $6$ variables are collected in the $3\times 2$ matrix:
\[
D_Xg = \left(\begin{array}{cc}
\partial_{11}g & \partial_{12}g\\
\partial_{21}g& \partial_{22}g \\
\partial_{31}g& \partial_{32}g \\
\end{array}
\right)
\]
The partial derivatives with respect to the remaining variables are denoted as $\partial_{(12,12)}g$, $\partial_{(13,12)}g$ and $\partial_{(23,12)}g$, i.e. following the ordering induced by $\Phi$. Finally, for a matrix $A \in \R^{r\times r}$, we denote with $\cof(A)$ the matrix defined as $$\cof(A)_{ij} = (-1)^{i + j}\det(M_{ji}(A)),$$ where $M_{ji}(A)$ denotes the $(n-1)\times(n-1)$ submatrix of $A$ obtained by eliminating from $A$ the $j$-th row and the $i$-th column. In particular, the following relation holds
\[
\cof(A)A = A\cof(A) = \det(A)\id_r.
\] 
We are ready to state the following:

\begin{prop}\label{p:convexity}
Let $f:\R^{n\times m}\to \R$ be a strictly polyconvex function of the form $f(X) =g(\Phi(X))$, where
$g \in C^1$ is strictly convex and $\Phi$ is the vector of all the subdeterminants of $X$, i.e.
\[
\Phi(X) = (X,v_1(X),\dots,v_{\min(n,m)}(X)),
\]
and $$v_s(X) = (\det(X_{Z_1}),\dots, \det(X_{Z_{\#\mathcal{A}_s}}))$$ for some fixed (but arbitrary) ordering of all the elements $Z\in\mathcal{A}_s$. If $A_i \in K'_f$ and $A_i \neq A_j$ for $i \neq j$, then $X_i$, $Y_i = D f(X_i)$ 
and $c_i = f (X_i)$ fulfill the following inequalities for every $i\neq j$:
\begin{multline}\label{finalemdim}
c_i - c_j +\langle Y_i,X_j - X_i\rangle - \sum_{r = 2}^{\min(m,n)}\sum_{Z\in\mathcal{A}_r}d^i_{Z}\left(\langle{\cof}(X_i^Z)^T,X^Z_j - X^Z_i\rangle -\det(X_j^Z) + \det(X_i^Z)\right)<0, 
\end{multline}
where $d^i_Z = \partial_Zg(\Phi(X_i))$. 
\end{prop}

This result was proved in \cite[Proposition 4.1]{DLDPKT}. We now introduce the set

\[
C'_f \doteq \left\{C'\in \R^{2n\times m}: C' =
\left(
\begin{array}{cc}
X\\ 
\beta Df(X)\\
\end{array}
\right), \text{ for some $\beta$ > 0}\right\}.
\]

Notice that $C'_f$ is the projection of $C_f$ on the first $2n\times m$ coordinates. We immediately obtain from the previous proposition and the definition of $C'_f$ that
\[
A_i \in C'_f, \quad \forall i\in \{1,\dots,N\}
\]
if and only if there exist numbers $\beta_i > 0, \forall i$, such that
\begin{multline}\label{finalemdimb}
c_i - c_j +\frac{1}{\beta_i}\langle Y_i,X_j - X_i\rangle -\sum_{r = 2}^{\min(m,n)}\sum_{Z\in\mathcal{A}_r}d^i_{Z}\left(\langle{\cof}(X_i^Z)^T,X^Z_j - X^Z_i\rangle -\det(X_j^Z) + \det(X_i^Z)\right)<0.
\end{multline}

The expressions in \eqref{finalemdimb} can be simplified when the matrices $X_1, \ldots , X_N$ induce a $T_N$ configuration:

\begin{corollary}\label{absurd}
Let $f$ be a strictly polyconvex function and let $A_1, \ldots , A_N$ be distinct elements of $K'_f$ with the additional property that $\{X_1, \ldots , X_N\}$ induces a $T_N$ configuration of the form \eqref{INCform} with defining vector $(\mu, \lambda)$. Then,
\begin{equation}\label{e:absurdb}
c_i - \sum_{j}t_j^ic_j -\frac{k_i}{\beta_i}\langle Y_i,C_i\rangle< 0 ,\quad \forall i \in \{1,\dots, N\},
\end{equation}
where the $t^i$'s are given by \eqref{defnt}. 
\end{corollary}

This corresponds to \cite[Corollary 4.3]{DLDPKT}, and concludes the list of preliminary results needed for the results of this paper.

\section{Positive case: proof of the main results}\label{MT}

Before checking whether the inclusion set $C_f$ contains $T_N$ or $T'_N$ configurations, we need to exclude more basic building block for wild solutions, such as rank-one connections or, as in this case, $\Lambda_{dc}$-connections in $C_f$. It is rather easy to see, compare for instance \cite{LSP}, that if $f$ is strictly polyconvex, then for $A,B \in K_f$ it is not possible to have
\[
A-B \in \Lambda_{dc}.
\]
Indeed the same result holds even considering $K'_f$. To prove this, it is sufficient to observe that if $X,Y \in \R^{n\times m}$ are rank-one connected, i.e. for some $u \in \mathbb{S}^{m - 1}$
\begin{equation}\label{r1}
(X - Y)v = 0,\; \forall v\perp u, 
\end{equation}
and
\begin{equation}\label{rn}
(Df(X) - Df(Y))u = 0,
\end{equation}
then
\begin{align*}
\langle Df(X) - Df(Y), X - Y \rangle &= \sum_{i = 1}^m ((Df(X) - Df(Y))u_i, (X-Y) u_i)\\
&\overset{\eqref{r1}}{=} ((Df(X) - Df(Y))u, (X-Y) u) \overset{\eqref{rn}}{=} 0,
\end{align*}
where $\{u_1,\dots, u_m\}$ is an orthonormal basis of $\R^m$ with $u_1 = u$. On the other hand, since $f$ is strictly polyconvex, it is easy to see that
\[
\langle Df(X) - Df(Y), X - Y \rangle > 0
\]
if $\rank(X-Y) = 1$. The first result of this section shows that this result holds also for $C_f$, provided $f$ is positive.

\begin{prop}\label{RANK}
Let $f$ be strictly polyconvex. If
\[
A =\left(
\begin{array}{cc}
X\\ 
Y\\
Z
\end{array}
\right),\; B= 
\left(
\begin{array}{cc}
X' \\ 
Y' \\
Z' 
\end{array}
\right)\in C_f,
\]
and $f(X) \ge 0, f(X') \ge 0$, then
\[
A- B \notin \Lambda_{dc}.
\]
\end{prop}
\begin{proof}
Suppose by contradiction that there exist
\[
A=
\left(
\begin{array}{cc}
X\\ 
Y\\
Z
\end{array}
\right) \in C_f,\;B=
\left(
\begin{array}{cc}
X' \\ 
Y' \\
Z' 
\end{array}
\right)
= \left(
\begin{array}{cc}
X + C\\ 
Y + D\\
Z + E
\end{array}
\right) \in C_f,
\]
with $c\doteq f(X) \ge 0, c'\doteq f(X') \ge 0$, and there is a vector $\xi \in \R^m$ with $\|\xi\| = 1$ such that for every $v\perp \xi$, $$Cv = 0,\;D\xi =0,\; E\xi = 0.$$ Now we can use the so-called Matrix Determinant Lemma \ref{MDL} to see that the expressions found in \eqref{finalemdimb} evaluated at $$A_1 = \left(\begin{array}{c}X\\Y\end{array}\right), \; A_2 = \left(\begin{array}{c}X + C\\Y + D\end{array}\right),$$ yield the following inequalities:
\begin{align}
&c - c' - \frac{1}{\beta}\langle X- X', Y\rangle< 0,\label{cc'} \\
&c' - c - \frac{1}{\beta'}\langle X' - X, Y'\rangle< 0.\label{c'c}
\end{align}
Moreover by assumption $(Z' - Z)\xi = 0$, i.e.
\[
(Z' - Z)\xi = 0 = (X')^TY'\xi - X^TY\xi -( c'\beta' - c\beta)\xi. 
\]
Thus, using $(Y' - Y)\xi = 0$,
\[
0 = (X' - X)^TY'\xi-( c'\beta' - c\beta)\xi = \langle C,Y\rangle\xi -( c'\beta' - c\beta)\xi,
\]
that yields, since $\|\xi\| = 1$,
\begin{equation}\label{CC}
\langle C,Y\rangle = c'\beta' - c\beta.
\end{equation}
In the previous lines we have used the fact that
\[
(X' - X)^TY'\xi = C^T(Y + D)\xi = C^TY\xi,
\]
and, since $C$ is of rank one with $Cv = 0, \forall v\perp \xi$,
\[
C^TY\xi = \langle C, Y\rangle\xi.
\]
Exploiting \eqref{CC}, we rewrite \eqref{cc'} as
\begin{equation}\label{one}
c - c' - \frac{1}{\beta}\langle X- X', Y\rangle = c - c' + \frac{1}{\beta}\langle C, Y\rangle = c - c' + \frac{1}{\beta} ( c'\beta' - c\beta)  < 0,
\end{equation}
and \eqref{c'c} as
\begin{equation}\label{two}
c' - c - \frac{1}{\beta'}\langle C, Y\rangle = c' - c - \frac{1}{\beta'}( c'\beta' - c\beta) < 0
\end{equation}
From \eqref{one}, we infer
\[
\beta c - \beta c' + (c'\beta' - c\beta) < 0 \Leftrightarrow  c'(\beta' - \beta) < 0
\]
and from \eqref{two}
\[
\beta' c' - \beta' c - (c'\beta' - c\beta) < 0 \Leftrightarrow c(\beta - \beta')<0.
\]
Since $c \ge 0$ and $c' \ge 0$, we get a contradiction.
\end{proof}

Let us recall the Matrix Determinant Lemma used in the proof of the last proposition:

\begin{lemma}\label{MDL}
Let $A,B$ be matrices in $\R^{m\times m}$, and let $\rank(B) \le 1$. Then,
\[
\det(A + B) = \det(A) + \langle\cof(A)^T,B\rangle.
\]
\end{lemma}

Now that we have excluded $\Lambda_{dc}$-connections, we can ask ourselves the same question concerning $T'_N$ configurations. In particular we want to prove the main Theorem of this part of the paper:

\begin{theorem}\label{t:pos}
If $f\in C^1 (\R^{n\times m})$ is a strictly polyconvex function, then $C_f$ does not contain any set $\{A_1, \ldots , A_N\} \subset \R^{(2n + m)\times m}$ which induces a $T_N'$ configuration, provided that $f(X_1) \ge 0,\dots, f(X_N) \ge 0$, if
\[
A_i = \left(
\begin{array}{c}
X_i\\
Y_i\\
Z_i
\end{array}
\right), \quad X_i,Y_i \in \R^{n\times m}, Z_i \in \R^{m\times m}, \forall i \in \{1,\dots, N\}.
\]
\end{theorem}

At the end of the section we will show the following

\begin{corollary}\label{t:main}
If $f\in C^1 (\R^{n\times m} )$ is strictly polyconvex, then $K_f$ does not contain any set $\{A_1, \ldots , A_N\}$ which induces a $T_N'$ configuration.
\end{corollary}

Let us fix the notation. We will always consider $T'_N$ configurations of the following form:

\begin{equation}\label{mainform}
A_i\doteq  \left(
\begin{array}{cc}
X_i\\ 
Y_i\\
Z_i\\
\end{array}
\right),\quad \; X_i,Y_i\in \R^{n\times m}, Z_i \in \R^{m\times m},
\end{equation}
with:
\begin{equation}\label{fixnot}
X_i = P + \sum_{j = 1}^{i - 1}C_j + k_iC_i, \; Y_i = Q + \sum_{j = 1}^{i - 1}D_j + k_iD_i, \;  Z_i = R + \sum_{j = 1}^{i - 1}E_j + k_iE_i,
\end{equation}
and we denote with $n_i \in \mathbb{S}^{m - 1}$ the vectors such that
\[
D_in_i = 0, E_in_i = 0, C_iv = 0,\quad \forall v \perp n_i,\; \forall 1\le i \le N.
\]

\subsection{Idea of the proof}\label{IDEA}

Before proving the theorem, let us give an idea of the key steps of the proof. First of all, in Lemma \ref{INCsimp}, we will see that without loss of generality we can choose $P = 0$. As already explained in Subsection \ref{STRA}, we want to prove that the system of inequalities
\begin{equation}\label{maininequalities}
\nu_i \doteq  \beta_ic_i - \sum_{j}\beta_it_j^ic_j -k_i\langle Y_i,C_i\rangle< 0, \forall i \, ,
\end{equation}
cannot be fulfilled at the same time. This gives a contradiction with Corollary \ref{absurd}. In particular, we show that for the index $\sigma$ such that $\beta_{\sigma} = \min_j\beta_j $,
\[
\nu_{\sigma} \ge 0.
\] 
To do so, we prove that the quantities
\begin{equation}\label{mui}
\mu_i\doteq  -\beta_ic_i + \sum_{j}\beta_jt_j^ic_j +k_i\langle Y_i,C_i\rangle
\end{equation}
equal to $0$ for every $i$. Then, choosing $\sigma$ as above and exploiting the positivity of $c_j, \forall j$, we estimate
\begin{equation}\label{contradiction2}
0 = -\mu_{\sigma} = \beta_{\sigma}c_{\sigma} - \sum_{j}\beta_jt_j^{\sigma}c_j -k_{\sigma}\langle Y_{\sigma},C_{\sigma}\rangle \le  \beta_{\sigma}c_{\sigma} - \sum_{j}\beta_{\sigma}t_j^{\sigma}c_j -k_{\sigma}\langle Y_{\sigma},C_{\sigma}\rangle = \nu_{\sigma}.
\end{equation}
This will then yield the required contradiction. In order to show $\mu_i = 0,\forall 1\le i \le N$, we consider $N$ matrices $M_i$ defined as
\[
M_i \doteq  \mu\sum_{j \le i -1}\alpha_jC_j^TD_j + \sum_{j \ge i}\alpha_jC_j^TD_j,
\]
where $\mu > 1$ is part of the defining vector of the $T_N$ configuration $\{X_1,\dots,X_N\}$, compare \ref{defvec}, and $\alpha_j$ are real numbers. We prove that for numbers $\xi_j > 0$, a subset $\mathcal{I}_i \subset \{\xi_1\mu_1,\dots, \xi_N\mu_N\}$ is made of generalized eigenvalues of $M_i$, see \eqref{gen}. This is achieved thanks to Lemma \ref{TEC}. Since $M_i$ is trace-free, as can be seen by the structure of $C_j$ and $D_j$, we will find $N$ relations of the form
\[
\sum_{\xi_j\mu_j \in \mathcal{I}_i}\xi_j\mu_j = 0.
\]
This can be read as the equations for the kernel for a specific matrix $N\times N$ matrix, $W$. Proving that $W$ has trivial kernel will yield $\xi_j\mu_j = 0,\forall j$, and thus $\mu_j = 0$ since $\xi_j > 0$. The proof of the invertibility of $W$ is the content of the last Lemma \ref{W}.

\subsection{Proof of Theorem \ref{t:pos}}

\begin{lemma}\label{INCsimp}
If $f$ is a strictly polyconvex function such that $A_i \in C_f$, $\forall 1\le i \le N$ and $f(X_i) \ge 0, \forall 1\le i\le N$, then there exists another strictly polyconvex function $F$ such that the $T_N'$ configuration $B_i$ defined as
\[
B_i =\left(
\begin{array}{cc}
X_i - P\\ 
Y_i\\
Z_i - P^TY_i\\
\end{array}
\right)
\]
satisfies $B_i \in C_F,$ for every $1\le i \le N$ and moreover $F(X_i - P) \ge 0, \forall i$.
\end{lemma}
\begin{proof}
Simply define the new polyconvex function $F(X)$ by $F(X)\doteq f(X + P)$. Clearly the newly defined family $\{B_1,\dots B_N\}$ still induces a $T'_N$ configuration, and it is straightforward that $B_i \in C_F$. Moreover, this does not affect positivity, in the sense that $F(X_i - P) = f(X_i - P + P) = f(X_i) \ge 0$.
\end{proof}

\begin{lemma}\label{TEC}
Suppose $A_i \in C_f$, $\forall i$, and $P = 0$. Then, for every $i \in \{1,\dots, N\}$:
\[
\sum_{j = 1}^N k_j(k_j - 1)t^i_jC_j^TD_jn_i = \left(k_i\langle C_i,Y_i\rangle - \beta_ic_i  + \sum_{j = 1}^N\beta_jt_j^ic_j\right)n_i \overset{\eqref{mui}}{=} \mu_in_i,\quad \forall i = 1,\dots, N,
\]
where $t^i$ is the vector defined in \eqref{defnt}.
\end{lemma}

\begin{proof}
We need to compute the following sums:
\begin{equation}\label{finalsum?}
\sum_jt_j^iZ_j = \sum_jt_j^iX^T_jY_j - \sum_jt_j^ic_j\beta_j\id.
\end{equation}
Let us start computing the sum for $i = 1$, $\sum_j\lambda_jX_j^TY_j.$ First, notice that
\[
 \sum_j\lambda_j X_j^TY_j =  \sum_j\lambda_j X_j^T(Y_j-Q) + \sum_j\lambda_j X_j^TQ = \sum_j\lambda_j X_j^T(Y_j-Q),
\]
since, by Lemma \ref{l:linear} or \eqref{e:sum_minor},
\[
\sum_j \lambda_j X_j^TQ = P^TQ = 0.
\]
We rewrite it in the following way:

\begin{equation}\label{quadsum}
\begin{split}
 \sum_j\lambda_j X_j^TY_j &=  \sum_j\lambda_j X_j^T(Y_j-Q) \\
&= \sum_{j = 1}^N\lambda_j\left(\sum_{1\le a,b\le j - 1}C_a^TD_b + k_j\sum_{1\le a \le j - 1} C_a^TD_j + k_j\sum_{1\le b \le j - 1} C_j^TD_b + k_j^2C_j^TD_j\right) \\
&=\sum_{i,j}g_{ij}C_i^TD_j,
\end{split}
\end{equation}
where we collected in the coefficients $g_{ij}$ the following quantities:
\[
g_{ij}=
\begin{cases}
 \lambda_ik_i + \sum_{r = i + 1}^N\lambda_r,\text{ if } i \neq j\\
\lambda_ik_i^2 + \sum_{r = i + 1}^N\lambda_r,\text{ if } i = j.
\end{cases}
\]
Using \eqref{defnk}, we have, if $i \neq j$:
\[
g_{ij} = g_{ji} = \lambda_ik_i + \sum_{r = i + 1}^N\lambda_r = \frac{\mu}{\mu - 1}.
\]
On the other hand, again using \eqref{defnk},
\[
g_{ii} = k_i^2\lambda_i + \sum_{r = i + 1}^N\lambda_r = k_i(k_i - 1)\lambda_i + \frac{\mu}{\mu - 1}.
\]
Using the equalities $\sum_\ell C_\ell = 0 = \sum_\ell D_\ell$, then also $\sum_{i,j}C_i^TD_j = 0$, and so $\sum_{i\neq j}C_i^TD_j = -\sum_iC_i^TD_i$. Hence, \eqref{quadsum} becomes
\[
 \sum_{i,j}g_{ij}C_i^TD_j = \frac{\mu}{\mu - 1}\sum_{i\neq j}C_i^TD_j + \sum_i\left( k_i(k_i - 1)\lambda_i + \frac{\mu}{\mu - 1}\right)C_i^TD_i =  \sum_i k_i(k_i - 1)\lambda_iC_i^TD_i.
\]
We just proved that
\begin{equation}\label{b}
\sum_j\lambda_jX_j^TY_j =  \sum_j k_j(k_j - 1)\lambda_jC_j^TD_j.
\end{equation}
Recall the definition of $t^i$, namely 
$$t^i= \frac{1}{\xi_i}(\mu\lambda_1,\dots,\mu\lambda_{i - 1},\lambda_i,\dots,\lambda_N)\, .$$
By the previous computation ($i = 1$), it is convenient to rewrite \eqref{finalsum?} using \eqref{b} as
\begin{equation}\label{rew}
R + \sum_{j=1}^{i - 1}E_j = \frac{1}{\xi_i}\left(\sum_j k_j(k_j - 1)\lambda_jC_j^TD_j + (\mu - 1)\sum_{j = 1}^{i - 1}\lambda_jX_j^TY_j\right) - \sum_jt_j^ic_j\beta_j\id.
\end{equation}
In the previous equation, we have used the equality
\begin{equation}\label{Zsum}
\sum_{j = 1}^Nt_j^iZ_j = R + \sum_{j=1}^{i - 1}E_j,\quad \forall i \in \{1,\dots, N\},
\end{equation}
that easily follows from Lemma \ref{l:linear}. Once again, let us express the sum up to $i - 1$ in the following way:
\[
\sum_{j = 1}^{i - 1}\lambda_jX_j^TY_j = \sum_{j = 1}^{i - 1}\lambda_jX_j^TQ + \sum_{j = 1}^{i - 1}\lambda_jX_j^T(Y_j-Q) = \sum_{j = 1}^{i - 1}\lambda_j X_j^TQ+ \sum_{k,j}^{i - 1}s_{kj}C_k^TD_j.
\]
A combinatorial argument analogous to the one in the previous case gives
\begin{align*}
&
s_{\ell\ell}= k_\ell^2\lambda_\ell + \dots + \lambda_{i - 1}
 \\
&\phantom{s_{\ell\ell}}= (k_\ell^2 - k_\ell)\lambda_\ell + k_\ell\lambda_\ell  + \dots + \lambda_{i - 1},
\\
& s_{\alpha\beta} = k_\alpha\lambda_\alpha + \dots + \lambda_{i - 1},\qquad\alpha\neq\beta\\
\end{align*}
Now
\[
 k_r\lambda_r + \dots + \lambda_{i - 1}=\frac{\mu( \sum_{j = 1}^{i - 1}\lambda_j) + \sum_{j = i}^{N}\lambda_j }{\mu - 1}
\]
and so
\[
 k_r\lambda_r + \dots + \lambda_{i - 1}=\frac{(\mu -1)( \sum_{j = 1}^{i - 1}\lambda_j) + 1 }{\mu - 1} = \frac{\xi_i}{\mu - 1} =: b_{i - 1}
\]
Hence
\[
\sum_{j = 1}^{i - 1}\lambda_jX_j^TY_j = \sum_{j = 1}^{i - 1}\lambda_j X_j^TQ+ \sum_{k,j}^{i - 1}s_{kj}C_k^TD_j =\sum_{j = 1}^{i - 1}\lambda_j X_j^TQ + b_{i - 1}\sum_{k,j}^{i - 1}C_k^TD_j + \sum_{j = 1}^{i - 1} k_j(k_j - 1)\lambda_j C_j^TD_j.
\]
We rewrite \eqref{rew} as
\begin{equation}\label{useful}
\begin{split}
R + \sum_{j = 1}^{i - 1}E_j&= \frac{1}{\xi_i}\left(\sum_j k_j(k_j - 1)\lambda_jC_j^TD_j + \xi_i\sum_{k,j}^{i - 1}C_k^TD_j +  (\mu - 1)\sum_{j = 1}^{i - 1} (k_j(k_j - 1)\lambda_j C_j^TD_j + \lambda_j X_j^TQ)\right) \\
&- \sum_j\beta_jt_j^ic_j\id
\end{split}
\end{equation}
Now we substitute \eqref{useful} in the definition \eqref{fixnot} of $Z_i$ in order to compute $E_i$:
\begin{align*}
k_iE_i &+ \frac{1}{\xi_i}\left(\sum_j k_j(k_j - 1)\lambda_jC_j^TD_j + \xi_i\sum_{k,j}^{i - 1}C_k^TD_j +  (\mu - 1)\sum_{j = 1}^{i - 1} (k_j(k_j - 1)\lambda_j C_j^TD_j + \lambda_j X_j^TQ)\right) \\
&- \sum_j\beta_jt_j^ic_j\id = X_i^TY_i - \beta_ic_i\id.
\end{align*}
Multiply by $n_i$ the previous expression and recall that $E_in_i = 0$ to find:
\begin{equation}\label{sp}
\begin{split}
\frac{1}{\xi_i}&\left(\sum_j k_j(k_j - 1)\lambda_jC_j^TD_jn_i + \xi_i\sum_{k,j}^{i - 1}C_k^TD_jn_i +  (\mu - 1)\sum_{j = 1}^{i - 1} (k_j(k_j - 1)\lambda_j C_j^TD_j n_i + \lambda_j X_j^TQn_i)\right) \\
&- \sum_j\beta_jt_j^ic_jn_i =  X_i^TY_in_i - \beta_ic_in_i.
\end{split}
\end{equation}
Now notice that, since $D_in_i = 0$,
\[
\begin{split}
X_i^TY_in_i &= X_i^TQn_i + \sum_{j,k}^{i- 1}C_{k}^TD_jn_i + k_i\sum_{j = 1}^{i-1}C_i^TD_jn_i + k_i\sum_{j = 1}^{i-1}C_j^TD_in_i + k_i^2C_i^TD_in_i \\
&= X_i^TQn_i + \sum_{j,k}^{i- 1}C_{k}^TD_jn_i + k_i\sum_{j = 1}^{i-1}C_i^TD_jn_i.
\end{split}
\]
Thus \eqref{sp} becomes
\begin{equation}\label{sp1}
\begin{split}
\frac{1}{\xi_i}&\left(\sum_j k_j(k_j - 1)\lambda_jC_j^TD_jn_i +  (\mu - 1)\sum_{j = 1}^{i - 1} (k_j(k_j - 1)\lambda_j C_j^TD_j n_i + \lambda_j X_j^TQn_i)\right) \\
&- \sum_j\beta_jt_j^ic_jn_i = X_i^TQn_i + k_i\sum_{j = 1}^{i-1}C_i^TD_jn_i - \beta_ic_in_i.
\end{split}
\end{equation}
Now we need to compute
\[
\sum_{j = 1}^{i - 1}\lambda_j X_j = \sum_{j = 1}^{i- 1}y_jC_j,
\]
and
\[
y_j = k_j\lambda_j + \dots + \lambda_{i - 1} = \frac{\xi_i}{\mu - 1}, \; \forall j \in \{1,\dots, i - 1\}.
\]
Using this computation, \eqref{sp1} reads as:
\begin{equation}\label{good}
\begin{split}
\frac{1}{\xi_i}&\left(\sum_j k_j(k_j - 1)\lambda_jC_j^TD_jn_i +  (\mu - 1)\sum_{j = 1}^{i - 1} k_j(k_j - 1)\lambda_j C_j^TD_j n_i\right) + \sum_{j = 1}^{i - 1}C_j^TQn_i - \sum_j\beta_jt_j^ic_jn_i = \\
&X_i^TQn_i + k_i\sum_{j = 1}^{i-1}C_i^TD_jn_i - \beta_ic_in_i.
\end{split}
\end{equation}
Exploiting the definition of $t_j^i$, we see that we can rewrite
\[
\frac{1}{\xi_i}\left(\sum_j k_j(k_j - 1)\lambda_jC_j^TD_jn_i +  (\mu - 1)\sum_{j = 1}^{i - 1} k_j(k_j - 1)\lambda_j C_j^TD_j n_i\right) = \sum_j k_j(k_j - 1)t^i_jC_j^TD_jn_i,
\]
and
\begin{align*}
X_i^TQn_i &+ k_i\sum_{j = 1}^{i-1}C_i^TD_jn_i - \sum_{j = 1}^{i - 1}C_j^TQn_i \\
&= \sum_{j = 1}^{i - 1}C_j^TQn_i + k_iC^T_iQn_i + k_i\sum_{j = 1}^{i-1}C_i^TD_jn_i - \sum_{j = 1}^{i - 1}C_j^TQn_i = k_iC_i^TY_in_i.
\end{align*}
Thus \eqref{good} becomes
\begin{align*}
\sum_j k_j(k_j - 1)t^i_jC_j^TD_jn_i - \sum_j\beta_jt_j^ic_jn_i = k_i\sum_{j = 1}^{i-1}C_i^TY_in_i - \beta_ic_in_i.
\end{align*}
Since $C_iv = 0, \forall v\perp n_i$, we have $C_i^TY_in_i = \langle C_i,Y_i\rangle n_i$, and we finally obtain the desired equalities:
\[
\sum_{j = 1}^N k_j(k_j - 1)t^i_jC_j^TD_jn_i = \left(k_i\langle C_i,Y_i\rangle - \beta_ic_i  + \sum_{j = 1}^N\beta_jt_j^ic_j\right)n_i, \quad \forall i = 1,\dots, N.
\]

\end{proof}

We are finally in position to prove the main Theorem.

\begin{proof}[Proof of Theorem \ref{t:pos}]
Assume by contradiction the existence of a $T_N'$ configuration induced by matrices $\{A_1, \ldots , A_N\}$ of the form \eqref{mainform} which belong to the inclusion set $C_f$ of some stictly polyconvex function $f\in C^1 (\R^{n\times m})$ and $f(X_i) \ge 0$ for every $i$. We can assume, without loss of generality by Lemma \ref{INCsimp}, that 
\[
P = 0\, .
\]
\noindent Using Lemma \ref{TEC}, we find
\begin{equation}\label{INCeigen}
\sum_{j = 1}^N k_j(k_j - 1)t^i_jC_j^TD_jn_i = \left(k_i\langle C_i,Y_i\rangle - \beta_ic_i  + \sum_{j = 1}^N\beta_jt_j^ic_j\right)n_i = \mu_in_i,\quad \forall i.
\end{equation}
Define $\alpha_{j}\doteq  k_j(k_{j} -1)\lambda_j > 0$, and
\[
M_i\doteq  \mu\sum_{j \le i -1}\alpha_jC_j^TD_j + \sum_{j \ge i}\alpha_jC_j^TD_j,
\]
for $i \in \{1,\dots, N\}$. Also set
\[
M_i\doteq  \mu M_{i - N},\quad \forall i \in \{N + 1,\dots, 2N\}.
\]
Then, \eqref{INCeigen} can be rewritten as
\begin{equation}\label{icase}
M_in_i = \xi_i\mu_in_i,\quad \forall i \in \{1,\dots,N\}.
\end{equation}
We define $n_{s}\doteq n_{s - N}$, for $s \in \{N + 1,\dots, 2N\}$. As explained in Subsection \ref{IDEA}, the rough idea is to show that a subset of the vectors $n_j$ are \emph{generalized eigenvectors}  and a subset of \emph{$\xi_j \mu_j$ are generalized eigenvalues} of $M_i$. In particular, for every $i \in \{1,\dots,N\}$, we want to show the following equalities:
\begin{equation}\label{gen}
\begin{cases}
M_in_{i + a} = \xi_{i + a}\mu_{i + a}n_{i + a} + v_{i,a}, &\text{ if } a: i \le i + a \le N\\
M_in_{i + a} = \mu\xi_{i + a}\mu_{i + a}n_{i + a} + v_{i,a}, &\text{ if } a: N + 1 \le i + a \le N + i - 1,
\end{cases}
\end{equation}
where $v_{i,a} \in \spn\{n_i,\dots, n_{i + a - 1}\}$. From now on, we fix $i \in \{1,\dots, N\}$. To prove \eqref{gen}, first we rewrite
\begin{equation}
M_{i}n_{i + a} = (M_{i} - M_{i + a})n_{i + a} + M_{i + a}n_{i + a},
\end{equation}
and then we use \eqref{icase} to obtain
\[
(M_{i} - M_{i + a})n_{i + a} + M_{i + a}n_{i + a} = \begin{cases}
\xi_{i + a}\mu_{i + a}n_{i + a} + (M_{i} - M_{i + a})n_{i + a}, \text{ if } i + a \le N,\\
\mu\xi_{i + a}\mu_{i + a}n_{i + a} + (M_{i} - M_{i + a})n_{i + a}, \text{ if } i + a > N.
\end{cases}.
\]
To conclude the proof of \eqref{gen}, we only need to show that
\begin{equation}\label{via}
(M_{i} - M_{i + a})n_{i + a} \in \spn\{n_i,\dots, n_{i + a - 1}\}, \quad \forall a \in \{0,\dots, N-1\}.
\end{equation}
To do so, we compute $M_i - M_{i + a}$. Let us start from the case $1\le i + a \le N$:
\begin{align*}
M_{i} - M_{i + a} &= \mu\sum_{j < i}\alpha_jC_j^TD_j + \sum_{j \ge i}\alpha_jC_j^TD_j - \mu\sum_{j < i + a}\alpha_jC_j^TD_j - \sum_{j \ge i + a}\alpha_jC_j^TD_j\\
&= \sum_{i \le j < i +a}\alpha_jC_j^TD_j - \mu\sum_{i \le j < i + a}\alpha_jC_j^TD_j.
\end{align*}
On the other hand, if $N + 1 \le i + a \le i + N - 1$, then
\begin{align*}
M_{i} - M_{i + a} &= M_{i} - \mu M_{i + a - N} \\
&= \mu\sum_{j < i}\alpha_jC_j^TD_j + \sum_{j \ge i}\alpha_jC_j^TD_j - \mu^2\sum_{j < i + a - N}\alpha_jC_j^TD_j - \mu\sum_{j \ge i + a - N}\alpha_jC_j^TD_j\\
&= \mu\sum_{j <  i + a - N}\alpha_jC_j^TD_j - \mu\sum_{j \ge  i}\alpha_jC_j^TD_j + \sum_{j \ge i}\alpha_jC_j^TD_j - \mu^2\sum_{j < i + a - N}\alpha_jC_j^TD_j.
\end{align*}
Now the crucial observation is that, due to the fact that of $C_jv = 0$ for every $v\perp n_j$, the image of $C_j^TD_j$ is contained in the line $\spn(n_j)$, for every $j \in \{1,\dots, N\}$. Therefore, the previous computations prove \eqref{via} and hence \eqref{gen}. Now we introduce
\[
V_i\doteq  \{n_i,n_{i + 1}, n_{i + 2},\dots, n_{N}, n_{N + 1}, \dots, n_{N + i - 1}\}.
\]
We can extract a basis for $\spn(V_i)$ in the following way. First, choose indexes
\begin{equation}\label{S}
\overline S_i\doteq \{k: k = i \text{ or } i< k \le N + i - 1, n_k \notin \spn(n_i,\dots,n_{k - 1})\}.
\end{equation}
Then, consider the basis $\mathcal{B}_i\doteq \{n_k: k \in \overline{S}_i\}$ for $\spn(V_i)$. Since $$\spn(\mathcal{B}_i) = \spn(\{n_1,\dots, n_N\}),\quad \forall i,$$ then $\#S_i = C \le \min\{m,N\}, \forall i$. Indexes in $\overline S_i$ lie in the set $\{1,\dots, 2N\}$. For technical reasons, we also need to consider the \emph{modulo $N$} counterpart of $\overline S_i$, that is
\begin{equation}\label{SNO}
S_i\doteq  \{k \in \{1,\dots, N\}: k \in \overline{S}_i \text{ or } k + N \in \overline{S}_i\}.
\end{equation}
In $S_i$, consider furthermore $S_i'\doteq  S_i \cap \{i,\dots, N\}$, $S_i''\doteq  S_i\cap\{1,\dots,i - 1\}$. If necessary, complete $\mathcal{B}_i$ to a basis of $\R^m$ made with elements $\gamma_j$ orthogonal to the ones of $\mathcal{B}_i$. Note that, since $\im(C_i^TD_i) \subset \spn(n_i)$, then $\im(M_i)\subset \{n_1,\dots, n_N\}$. Then, the associated matrix to $M_i$ with respect to $\mathcal{B}_i$ is
\begin{equation}\label{Mform}
M_i=
\left(
\begin{array}{c|c}
\begin{matrix}
    a_{1i} &* &* & \dots  &* \\
    0 &  a_{2i} &* & \dots  &* \\
    \vdots & \vdots & \vdots & \ddots & \vdots \\
    0 & 0 &0 & \dots  &   a_{Ci}
\end{matrix}

& \mathbf{T}\\ \hline
\mathbf{0}_{m - C,C} & \mathbf{0}_{m - C,m - C}
\end{array}
\right).
\end{equation}
We denoted with $\mathbf{0}_{c,d}$ the zero matrix with $c$ rows and $d$ columns, with $\mathbf{T}$ the $C\times (m - C)$ matrix of the coefficients of $M_i\gamma_j$ with respect to $\{n_s:s\in \overline{S}_i\}$, and with $\mathbf{*}$ numbers we are not interested in computing explicitely. Finally, we have chosen an enumeration $s_1< s_2<\dots < s_\ell < \dots < s_C$ of the elements of $\overline{S}_i$, and we have defined
\[
a_{\ell i}= \begin{cases}
\xi_{s_{\ell}}\mu_{s_{\ell}}, &\text{ if } s_{\ell} \in S_i',\\
\mu\xi_{s_{\ell}-N}\mu_{s_{\ell}-N}, &\text{ if } s_{\ell} - N \in S_i''.
\end{cases}
\]
The triangular form of the matrix representing $M_i$ is exactly due to \eqref{gen}. Now, $\tr(M_i) = 0, \forall i$, since $C_i^TD_i$ is trace-free for every $i$. This implies that the matrix in \eqref{Mform} must be trace-free, hence:
\begin{equation}\label{ker}
0 = \tr(M_i) = \sum_{\ell = 1}^C a_{\ell i} = \sum_{a \in S'_i}\xi_{a}\mu_{a} + \mu \sum_{b \in S''_i}\xi_{b}\mu_{b}.
\end{equation}
We have thus reduced the problem to the following simple Linear Algebra statement: we wish to show that, if $W$ is the $N\times N$ matrix defined as
\[
W_{ij}=
\begin{cases}
1, & \text{ if } j \in S_i',\\
\mu, & \text{ if } j \in S_i'',\\
0, & \text{ if } j \notin S_i,
\end{cases}
\]
then, $Wx = 0 \Rightarrow x = 0$. By \eqref{ker}, the vector $x \in \R^{N}$ defined as $x_j\doteq  \xi_j\mu_j, \forall 1\le j \le N$, is such that $Wx = 0$, thus if the statement is true we get $\xi_j \mu_j = 0, \forall 1\le j \le N$, and since $\xi_j > 1$, also $\mu_j = 0,\forall 1 \le j \le N$. By \eqref{contradiction2}, this is sufficient to reach a contradiction. Therefore, we only need to show that $Wx = 0 \Rightarrow x = 0$. This proof will be given in Lemma $\ref{W}$.
\end{proof}

Before giving the proof of the final Lemma, let us make some examples of possible matrices $W$ arising from the previous construction. For the sake of illustration, let us take $N$ to be as small as possible, i.e. $N = 4$.

\begin{ex}\label{e:1}
Consider the case in which $C = 2$. This corresponds, for instance, to the case $m = 2$. Then, by Proposition \ref{RANK} and \eqref{S}, the only possible form of $W$ is
\[
W =
\left(
\begin{array}{c}
\begin{matrix}
    1 & 1 & 0 & 0\\
0 & 1 & 1 & 0 \\
0 & 0 & 1& 1\\
\mu & 0 & 0 & 1
\end{matrix}
\end{array}
\right),
\; Wx =
\left(
\begin{array}{cc}
x_1 + x_2\\
x_2 + x_3\\
x_3 + x_4\\
\mu x_1 + x_4
\end{array}
\right) = 0.
\]
Let $W_i$ be the $i$-th row of $W$. We notice that for $i = 1,2,3$, $W_{i + 1}$ differs from $W_i$ by exactly two elements, while $W_4$ does not differ with $W_1$ by only two elements. It does, though, with $\mu W_1$. Hence we rewrite equivalently the system $Wx = 0$ as $W_i - W_{i + 1}$, $W_4 - \mu W_1$:
\[
0 = \left(\begin{array}{cc}
x_1 - x_3\\
x_2 - x_4\\
x_3 - \mu x_1\\
x_4 - \mu x_2
\end{array}
\right), \text{ i.e. }
x_i = a_i x_{h(i)}, a_i = 
\begin{cases}
1, &\text{ if } h(i) > i,\\
\mu, &\text{ if } h(i) \le i,
\end{cases} 
\]
For a function $h: \{1,\dots, 4\} \to \{1,\dots, 4\}$. Since $\mu > 1$, this immediately implies $x_i = 0, \forall i$.
\end{ex}

\begin{ex}\label{e:2}
Consider the case in which $C = 4$, corresponding to $n_1,n_2,n_3,n_4$ linearly independent. Then,
\[
W =
\left(
\begin{array}{cc}
\begin{matrix}
    1 & 1 & 1 & 1\\
\mu & 1 & 1 & 1 \\
\mu & \mu & 1& 1\\
\mu & \mu & \mu & 1
\end{matrix}
\end{array}
\right),
\; Wx =
\left(
\begin{array}{cc}
x_1 + x_2 + x_3 + x_4\\
\mu x_1 + x_2 + x_3 + x_4\\
\mu x_1 + \mu x_2 + x_3 + x_4\\
\mu x_1 + \mu x_2 + \mu x_3 + x_4
\end{array}
\right) = 0.
\]
As in the previous example, for $i = 1,2,3$, $W_{i + 1}$ differs from $W_i$ by exactly one element, while $W_4$ does the same with $\mu W_1$. Thus as before we rewrite equivalently the system $Wx = 0$ as $W_i - W_{i + 1}$, $W_4 - \mu W_1$:
\[
0 = \left(\begin{array}{cc}
(\mu - 1)x_1\\
(\mu - 1)x_2\\
(\mu - 1)x_3\\
(\mu - 1)x_4
\end{array}
\right), \text{ i.e. }
x_i = a_i x_{h(i)}, a_i = 
\begin{cases}
1, &\text{ if } h(i) > i,\\
\mu, &\text{ if } h(i) \le i,
\end{cases} 
\]
In this case, $h(i) = i, \forall i \in \{1,\dots, 4\}$. Clearly also in this case $\mu > 1$, implies $x_i = 0, \forall i$.
\end{ex}

Finally, let us show a less symmetric example:

\begin{ex}\label{e:3}
Consider the case in which $C = 3$. Then, a possible matrix is:
\[
W =
\left(
\begin{array}{cc}
\begin{matrix}
    1 & 1 & 0 & 1\\
0 & 1 & 1 & 1 \\
\mu & 0 & 1& 1\\
\mu & \mu & 0 & 1
\end{matrix}
\end{array}
\right),
\; Wx =
\left(
\begin{array}{cc}
x_1 + x_2 + x_4\\
x_2 + x_3 + x_4\\
\mu x_1 + x_3 + x_4\\
\mu x_1 + \mu x_2 + x_4
\end{array}
\right) = 0.
\]

First, let us comment on the fact that this is a possible matrix appearing in the proof of the previous Theorem. Indeed, let us consider the first two lines:
\[
\left(
\begin{matrix}
1 & 1 & 0 & 1\\
0 & 1 & 1 & 1 \\
\end{matrix}
\right).
\]
The fact that $W_{13} = 0$ means that $n_3 \in \spn(n_1,n_2)$, since $3 \notin S_1$. On the other hand, Proposition \ref{RANK} ensures that $n_{3}$ is not a multiple of $n_2$, hence $n_3 \in S_2$, and $W_{23} = 1 \not = 0.$ For this reason, the matrix
\[
W =
\left(
\begin{array}{cc}
\begin{matrix}
    1 & 0 & 1 & 1\\
0 & 1 & 1 & 1 \\
\mu & 0 & 1& 1\\
\mu & \mu & 0 & 1
\end{matrix}
\end{array}
\right)
\]
would for instance have been non-admissible. Now, in order to prove $Wx = 0 \Rightarrow x = 0$, we work as in the previous examples, by noticing that for $i = 1,2,3$, $W_{i + 1}$ differs from $W_i$ by at most two elements, while $W_4$ must be compared with $\mu W_1$. Thus we write $W_i - W_{i + 1}$, $W_4 - \mu W_1$:
\[
0 = \left(\begin{array}{cc}
x_1 - x_3\\
x_2 - \mu x_1\\
x_3 - \mu x_2\\
(\mu - 1)x_4
\end{array}
\right), \text{ i.e. }
x_i = a_i x_{h(i)}, a_i = 
\begin{cases}
1, &\text{ if } h(i) > i,\\
\mu, &\text{ if } h(i) \le i.
\end{cases} 
\]
It is an elementary computation to show that $x_i = 0, \forall i$.
\end{ex}

Even though the examples we have given are too simple to appreciate the usefulness of the function $h$ such that $x_i = a_ix_{h(i)}$, this will be crucial in the proof of the Lemma.

\begin{lemma}\label{W}
Let $W$ be the matrix defined in the proof of Theorem \ref{t:pos}. Then, $\Ker(W) = \{0\}$.
\end{lemma}

\begin{proof}

Throughout the proof, we always consider a given vector $x \in \R^N$ such that $Wx = 0$. The proof, partially suggested by the previous examples, consists in the following steps. First, we show that the rows of $W$, $W_i$ and $W_{i + 1}$ (if $i = N$, we compare $W_N$ with $\mu W_1$) differ for at most two elements, and one of them is always $x_i$. This immediately yields the existence of a function $h: \{1,\dots, N\} \to \{1,\dots, N\}$ such that $x_i = a_ix_{h(i)}$. We will then use this and the crucial fact that $\mu > 1$ to conclude that $x_i = 0, \forall i$. Let us make the following claims, and see from them how to conclude the proof of the present Lemma. We will use freely the notation introduced at the end of the proof of Theorem \ref{t:pos}.
\\
\\
\indent\fbox{Claim $1$:} Let $i \in \{1,\dots, N\}$. Then $\overline{S}_i$ differs from $\overline{S}_{i + 1}$ (if $i = N$, $\overline{S}_{i + 1}= \overline{S}_1$) of at most two elements, in the sense that
\[
\overline{S}_i\Delta \overline{S}_{i + 1}\doteq  \overline{S}_i\setminus \overline{S}_{i + 1} \cup \overline{S}_{i+1}\setminus \overline{S}_{i} 
\]
contains at most 2 elements. Moreover, if $\overline{S}_i \Delta \overline{S}_{i + 1} \neq \emptyset$, then $\overline{S}_i \Delta \overline{S}_{i + 1} = \{i,I(i)\}$, with $i \in \overline{S}_{i}\setminus \overline{S}_{i +1}$, and $I(i) \in \overline{S}_{i + 1}\setminus \overline{S}_i$.
\\
\\
\indent\fbox{Claim 2:} Let $i \in \{1,\dots, N - 1\}$. The couple of rows $W_i$,$W_{i + 1}$ and $\mu W_1$, $W_N$ differ at most by two elements, in the sense that if $W_i = (W_{i1},\dots,W_{iN})$ and $W_{i + 1} = (W_{(i+1)1},\dots,W_{(i+1)N})$, then there are at most two indexes $j_1,j_2$ such that $W_{ij_1}-W_{(i + 1)j_1} \neq 0$ and $W_{ij_2}-W_{(i + 1)j_2} \neq 0$ (and analogously for $\mu W_1$ and $W_N$).
\\
\\
Finally, with this claim at hand, we are going to prove
\\
\\
\indent\fbox{Claim 3:} There exists a function $h: \{1,\dots, N\}\to \{1,\dots, N\}$ and numbers $a_i, i \in \{1,\dots, N\}$, such that
\begin{equation}\label{INCrel}
x_i = a_ix_{h(i)}
\end{equation}
with the property
\[
a_i = 
\begin{cases}
1, &\text{ if } h(i) > i,\\
\mu, &\text{ if } h(i) \le i.
\end{cases} 
\]

Let us show how the proof of the Lemma follows from Claim 3, and postpone the proofs of the claims. Fix $i \in \{1,\dots, N\}$ and use \eqref{INCrel} recursively to find
\[
x_i = a_ia_{h(i)}\dots a_{h^{(n - 1)}(i)}x_{h^{(n)}(i)},
\] 
where $h^{(n)}$ denotes the function obtained by applying $h$ to itself $n$ times. We also use the notation $h^{(0)}$ to denote the identity function: $h^{(0)}(i) = i$, $\forall i \in \{1,\dots,N\}$. By the properties of $a_{j}$, we have, $\forall r \in \{0,\dots, n - 1\}$,
\[
a_{h^{(r)}(i)} = 
\begin{cases}
1, &\text{ if } h^{(r)}(i) > h^{(r - 1)}(i),\\
\mu, &\text{ if } h^{(r)}(i) \le h^{(r - 1)}(i).
\end{cases}
\]
Fix $k \in \N$, and let $r \in \{k + 1, \dots, k + N + 1\}$. Then, $ h^{(r)}(i) > h^{(r - 1)}(i)$ can occur at most $N$ times in this range, since otherwise we would find
\[
1 \le h^{(k)}(i) <  h^{(k+ 1 )}(i) < h^{(k + 2)}(i) < \dots < h^{(k + N + 1)}(i) \le N,
\]
and this is impossible since we would have $N + 1$ distinct elements in the set $\{1,\dots, N\}$. Now clearly this observation implies that for every fixed $l \in \N$, there exists $s \in \N$ such that
\[
x_i = \mu^tx_{h^{(s)}(i)}, \text{ for some $t \ge l$}.
\]
This can only happen if $x_i = 0$. Since $i$ is arbitrary, the conclusion follows.
\\
\\
Let us now turn to the proof of the claims. 
\\
\\
\indent \fbox{Proof of claim $1$:}
To prove the claim, we need to use the definition of $\overline{S}_i$. Let us recall the definition of $\overline{S}_i$, given in \eqref{S}. To build $\overline{S}_i$ what we do is consider the ordered set $\{n_i,n_{i + 1},\dots, n_{i + 1 - N}\}$ and select from it a basis of $\spn\{n_1,\dots, n_N\}$ starting from $n_i$ and then at the step $1 \le k \le N-1$ deciding whether to insert the vector $n_{i + k}$ in our collection based on the fact that it is linear dependent or not from the previous ones. Recall also that $S_i$ is the \emph{modulo $N$} version of $\overline{S}_i$, see \eqref{SNO}, and that we define $n_j \doteq n_{j - N}$, for $j \in \{N + 1,\dots, 2N\}$. Hence now fix $i \in \{1,\dots, N\}$ and consider $S_i$. If $S_i = \{1,\dots, N\}$, then $\#S_i = N$, thus $S_j = \{1,\dots, N\}, \forall 1\le j \le N$ and the claim holds. Otherwise, let $i + 1 < I = I(i) \le i + N -1$ be the first element in $(\overline{S_i})^c$. There are two cases:

\begin{enumerate}
\item $n_{I} \in \spn (n_i,\dots, n_{I - 1}) \setminus \spn(n_{i + 1},\dots, n_{I - 1})$;
\item $n_{I} \in \spn (n_{i + 1},\dots, n_{I - 1})$.
\end{enumerate} 

At the same time, consider what happens in $\overline{S}_{i + 1}$: the span in the $(i + 1)$-th case starts with one vector less than the one of the $i$-th case, simply because the collection of indexes in $\overline{S}_{i + 1}$ starts from $n_{i + 1}$. Hence, since $I$ is the first missing index in $\overline{S}_i$, $I$ is also the first possible missing index for $\overline{S}_{i + 1}$. Therefore, consider the first case $$n_{I} \in \spn (n_i,\dots, n_{I - 1}) \setminus \spn (n_{i + 1},\dots, n_{I - 1}).$$ This implies that $I \in \overline{S}_{i + 1}$. Moreover, we are now adding $n_I$ to the set of vectors $n_{i+1},\dots, n_{I - 1}$, and $n_{I} \in \spn(n_i,\dots, n_{I - 1})\setminus \spn(n_{i + 1},\dots, n_{I - 1})$, hence $n_I$ adds to the previous vectors the component relative to $n_i$, in the sense that
\[
\spn(n_{i + 1},\dots,n_I) = \spn(n_i,\dots, n_{I - 1}).
\]
This moreover implies that $j \in \overline{S}_i \Leftrightarrow j \in \overline{S}_{i + 1}$, $\forall I \le j < N + i - 1$. Since $n_i \in \spn(n_{i + 1},\dots,n_I)$, $i \notin \overline{S}_{i + 1}$. Thus $\overline{S}_i$ and $\overline{S}_{i + 1}$ differ by at most two elements, and we have $i \in \overline{S}_{i}\setminus \overline{S}_{i + 1}$ and $I = I(i) \in \overline{S}_{i + 1}\setminus \overline{S}_i$. This concludes the case $$n_{I} \in \spn (n_i,\dots, n_{I - 1}) \setminus \spn(n_{i + 1},\dots, n_{I - 1}).$$ If instead $n_{I} \in \spn (n_{i + 1},\dots, n_{I - 1})$, then we see that $I \notin \overline{S}_{i + 1}$, and we can iterate this reasoning from there, in the sense that we look for the next index $I'$ such that ${I'}\notin \overline{S}_{i}$ and divide again into the two cases above. Clearly, for the indexes $i + 1 \le j < I'_1$, we have $j \in \overline{S}_{i + 1}$ and $j \in \overline{S}_i$. Either this iteration enters in case $1$ of the previous subdivision for some element $I \notin \overline{S}_i$, or we conclude $\overline{S}_i = \overline{S}_{i + 1}$. This concludes the proof of the claim.
\\
\\
\indent \fbox{Proof of claim $2$:}
\\
\\
Note that nonzero elements of $W_i$ are found in positions corresponding to elements of $S_i$. Hence now fix $i \in \{1,\dots, N- 1\}$ and consider $W_i$ and $W_{i + 1}$. If $S_i = S_{i + 1}$, then $W_{ij} = 0 \Leftrightarrow W_{(i + 1)j} = 0$. Moreover, we introduce the \emph{modulo $N$} counterpart of the number $I(i)$ found in Claim $2$, i.e. $I'(i) = I(i)$ if $I(i) \in \{1,\dots,N\}$, and $I'(i) = I(i) - N$ if $I(i) \in \{N + 1,\dots, 2N\}$. Thus using the definition of $W$, we can deduce, if $S_i = S_{i + 1}$,
\begin{equation}\label{ZD}
\begin{cases}
W_{(i+ 1)j} = W_{ij} = 0, &\text{ if } j \notin S_i\\
W_{(i+ 1)j} = W_{ij} = \mu, &\text{ if } j \in S_i, j < i\\
W_{(i+ 1)j} = W_{ij} = 1, &\text{ if } j \in S_i, j > i\\
W_{(i + 1)i} = \mu, W_{ii} = 1, & \text{ otherwise},
\end{cases}
\end{equation}
 and the claim holds in this case. 
Finally, if $S_i \Delta S_{i +1} = \{i,I'(i)\}$, then:
\begin{equation}\label{DD}
\begin{cases}
W_{(i+ 1)j} = W_{ij} = 0, &\text{ if } j \notin S_i,j \neq I'(i)\\
W_{(i+ 1)j} = 1, W_{ij} = 0, &\text{ if } j = I'(i) > i + 1\\
W_{(i+ 1)j} = \mu, W_{ij} = 0, &\text{ if } j = I'(i) < i -1\\
W_{(i+ 1)j} = W_{ij} = \mu, &\text{ if } j \in S_i, j < i\\
W_{(i+ 1)j} = W_{ij} = 1, &\text{ if } j \in S_i, j > i\\
W_{(i + 1)i} =0, W_{ii} = 1, & \text{ otherwise}.
\end{cases}
\end{equation}
This concludes the proof of the claim if $i \in \{1,\dots, N-1\}$. If $i = N$, then we need to compare $W_N$ with $\mu W_1$, and we obtain two cases, in analogy with the previous situation:
\begin{equation}\label{ZDN}
\text{if } S_N\Delta S_{1} = \emptyset, \text{ then}:
\begin{cases}
\mu W_{1j} = W_{Nj} = 0, &\text{ if } j \notin S_N\\
\mu W_{1j} = W_{Nj} = \mu, &\text{ if } j \in S_N, j < N\\
\mu W_{1N} = \mu, W_{NN} = 1, & \text{ otherwise},
\end{cases}
\end{equation}
and
\begin{equation}\label{UDN}
\text{if } S_N\Delta S_{1} = \{N,I'(N)\}, \text{ then}:
\begin{cases}
\mu W_{1j} = W_{Nj} = 0, &\text{ if } j \notin S_N, j \neq I'(N)\\
\mu W_{1j} = \mu, W_{Nj} = 0, &\text{ if } j \notin S_N, j = I'(N)\\
\mu W_{1j} = W_{Nj} = \mu, &\text{ if } j \in S_N, j < N\\
\mu W_{1N} = 0, W_{NN} = 1, & \text{ otherwise}.
\end{cases}
\end{equation}
\\
\\
\indent\fbox{Proof of Claim $3$:} Fix $i \in \{1,\dots, N\}$. We want to consider the equations given by
\[
(W_{i + 1}-  W_{i},x) =0, \text{ if } i\in\{1,\dots, N-1\}, \text{ and } (W_{N} - \mu W_{1},x) = 0.
\]
If we consider $i \in \{1,\dots,N -1\}$, we see from \eqref{ZD} and \eqref{DD} that
\[
0 = (W_{i}-W_{i + 1},x) = \sum_{j = 1}^N(W_{ij} - W_{(i+1)j})x_j = 
\begin{cases}
(1 - \mu)x_i,& \text{ if } S_i\Delta S_{i - 1} = \emptyset\\
x_i - x_{I'(i)},&\text{ if } S_i\Delta S_{i - 1} =\{i,I'(i)\}, I'(i) > i + 1\\
x_i - \mu x_{I'(i)},& \text{ if } S_i\Delta S_{i - 1} =\{i,I'(i)\}, I'(i) < i - 1
\end{cases}
\]
and from \eqref{ZDN} and \eqref{UDN} we infer
\[
0 = (W_{N}-\mu W_{1},x) =
\begin{cases}
(1 - \mu)x_N,& \text{ if } S_N\Delta S_{1} = \emptyset\\
x_N - \mu x_{I'(N)},&\text{ if } S_N\Delta S_{1} =\{N,I'(N)\}.
\end{cases}
\]
From these equations we see that \eqref{INCrel} holds with the choice $h(i)\doteq I'(i)$, when $i$ is such that $S_i\Delta S_{i + 1} \neq \emptyset$, and $h(i)\doteq  i$ otherwise.

\end{proof}

\subsection{Proof of Corollary \ref{t:main}}

We end this section by showing that Theorem \ref{t:pos} implies Theorem \ref{t:main}. Assume by contradiction that there exists a family of matrices $$\{A_1,\dots,A_N\} \subset K_f$$  inducing a $T'_N$ configuration of the form $\eqref{mainform}$. We show that then there exists another $T'_N$ configuration $\{B_1,\dots, B_N\}$ such that $B_i \in K_F \subset C_F, \forall 1\le i \le N$ for some strictly polyconvex $F$ with
\[
F(X'_i) \ge 0,\; \forall 1\le i \le N,
\]
if
\[
B_i= \left(
\begin{array}{cc}
X_i' \\ 
Y_i' \\
Z_i' 
\end{array}
\right),\; \forall 1 \le i \le N.
\]
This is a contradiction with Theorem \ref{t:pos}. To accomplish this, it is is sufficient to define $F(X)\doteq f(X) - \min_{i}f(X_i)$. This function is clearly strictly polyconvex, since $f$ is. Moreover, we define
\[
X_i'\doteq X_i,\; Y_i' \doteq Y_i \text{ and } Z_i'\doteq Z_i + \min_{j}f(X_j)\id.
\]
In this way, $B_i$ is still a $T_N'$ configuration. Moreover, $B_i \in K_F, \; \forall 1\le i \le N$. To see this, it is sufficient to notice that, since $A_i \in K_f$,
\[
Y_i' = Y_i = Df(X_i) = DF(X_i'),\;\forall 1\le i \le N,
\]
and
\[
Z_i' = Z_i +\min_{i}f(X_i)\id = X_i^TY_i - f(X_i)\id + \min_if(X_i)\id = (X'_i)^TY'_i - F(X_i)\id.
\]
This finishes the proof.

\section{Sign-changing case: the counterexample}\label{SCC}

In this section, we construct a counterexample to regularity in the case in which the hypothesis of non-negativity on $f$ is dropped. Let us explain the strategy, that follows the one of \cite{LSP}. First of all, we consider the following equivalent formulation of the differential inclusion of div-curl type considered in the previous sections. Indeed, due to the fact that for $a \in \Lip(\R^2,\R^2)$,
\[
\dv(a) = \curl(aJ),
\]
if
\[
J = \left( 
\begin{array}{ll}
0 & -1\\
1 & 0
\end{array}\right)\,,
\]
one easily sees that \eqref{vargrb} holds if and only if
\[
\left\{
\begin{array}{ll}
\displaystyle \curl(Df(D u)J) = 0,\\ 
\displaystyle \curl(Du^TDf(Du)J - f(Du)J)= 0,
\end{array}\right.
\]
in the weak sense. Since $\Omega$ is convex, the latter allows us to say that \eqref{vargrb} holds for $u \in \Lip(\Omega,\R^2)$ if and only if there exist $w_1,w_2:\Omega\to \R^2$ such that
\[
w\doteq \left(\begin{array}{c}u\\w_1\\w_2\end{array}\right)
\]
solves a.e. in $\Omega$:
\begin{equation}\label{curl}
Dw \in \tilde C_f \doteq \left\{C\in \R^{(2n + m)\times m}: C =
\left(
\begin{array}{cc}
X\\ 
\beta Df(X)J\\
\beta  X^TDf(X)J - \beta f(X)J
\end{array}
\right), \text{ for some $\beta$ > 0}\right\}.
\end{equation}
From now on, we will always use this reformulation of the problem. Let us also introduce
\[
\tilde C_f'\doteq  \left\{C\in \R^{(2n + m)\times m}: C =
\left(
\begin{array}{cc}
X\\ 
\beta Df(X)J
\end{array}
\right), \text{ for some $\beta$ > 0}\right\}.
\]
In order to construct the counterexample, we want to find a set of \emph{non-rigid} matrices $\{A_1,A_2,A_3,A_4,A_5\}$, $A_i \in \R^{6\times 2}, \forall i$, satisfying
\begin{equation}\label{Ai}
A_i = \left(
\begin{array}{cc}
X_i\\ 
Y_i\\
Z_i
\end{array}
\right) \in \tilde C_f.
\end{equation}
Roughly, non-rigidity means that there exists a non-affine solution of the problem
\[
Dw \in \{A_1,\dots, A_5\},
\]
see Lemma \ref{in-app}. The integrand $f$ is of the form
\begin{equation}\label{H}
f(X)= \varepsilon\mathcal{A}(X) + g(X,\det(X)),
\end{equation}
for some convex and smooth $g: \R^5 \to \R$ and
\begin{equation}\label{area}
\mathcal{A}(X) = a(X,\det(X)), \quad \text{ where } a(X,d) \doteq \sqrt{1 + \|X\|^2 + d^2},
\end{equation}
is the area function. As in \cite{LSP}, $f$ is not fixed from the beginning, but rather becomes another unknown of the problem. In particular, in order to find $f$, it is sufficient for the following condition to be fulfilled:
\begin{Cond}\label{c:1}
There exist $2\times 2$ matrices $\{X_1,\dots,X_5\}$, $\{Y_1,\dots, Y_5\}$, real numbers $c_1,\dots, c_5, d_1,\dots, d_5$ and positive integers $\beta_1,\dots, \beta_5$ such that for $Q_{ij}\doteq \displaystyle c_i - c_j + d_i\det(X_i - X_j) + \frac{1}{\beta_i}\langle X_i - X_j ,Y_iJ\rangle$, one has
\begin{equation}\label{prime-1}
Q_{ij} < 0, \forall i \neq j.
\end{equation}
\end{Cond}
If this condition is satisfied, then one has
\begin{equation}\label{prime}
\left(
\begin{array}{cc}
X_i\\ 
Y_i
\end{array}
\right) \in \tilde C'_f, \quad \text{i.e. } Y_i = \beta_iDf(X_i)J.
\end{equation}
The construction of $f$ is the content of Lemma \ref{l:1}. Moreover, we will be able to build $f$ in such a way that for some large $R > 0$, 
\begin{equation}\label{infty}
g(z) = M\sqrt{1 + \|z\|^2} - L = Ma(z) - L,\quad \forall z \in \R^5, \|z\| \ge R,
\end{equation}
and constants $M,L > 0$. The \emph{non-rigidity} of $A_1,\dots,A_5$ stems from the fact that we choose $\{X_1,\dots,X_5\}$ forming a \emph{large} $T_5$-configuration, in the terminology of \cite{FS}. Therefore we introduce:
\begin{Cond}\label{c:2}
$\{X_1,X_2,X_3,X_4, X_5\}$ form a large $T_5$ configuration, i.e. there exists at least three permutations $\sigma_1,\sigma_2,\sigma_3 : \{1,2,3,4,5\} \to \{1,2,3,4,5\}$ such that the ordered set $[X_{\sigma_i(1)},X_{\sigma_i(2)},\dots, X_{\sigma_i(5)}]$ is a $T_5$ configuration and moreover $\{C_{\sigma_1(i)},C_{\sigma_2(i)},C_{\sigma_3(i)}\}$ are linearly independent for every $i \in \{1,\dots,5\}$.
\end{Cond}

Once this condition is guaranteed, by \cite[Theorem 1.2]{FS}, we find a non-affine Lipschitz map $u:\Omega\subset \R^2 \to \R^2$ such that
\[
Du \in \{X_1,X_2,X_3,X_4,X_5\}
\]
almost everywhere in $\Omega$. Furthermore, we can choose $u$ with the property that for any subset $\mathcal{V}\subset \Omega$, $Du$ attains each of these matrices on a set of positive measure. This is proved in Lemma \ref{in-app}.
\\
\\
In order to find Lipschitz maps $w_1,w_2: \Omega \to \R^2$ such that $$w = \left(\begin{array}{c}u\\w_1\\w_2\end{array}\right): \Omega \to \R^6$$ satisfies
\[
Dw \in \tilde C_f\quad \text{a.e. in }\Omega,
\]
we simply consider $w_1 = Au + B$, $w_2 = Cu + D$, for suitable $2\times 2$ matrices $A,B,C,D$. We therefore get our last
\begin{Cond}\label{c:3}
$Y_i$ and $Z_i$ can be chosen of the form
\[
Y_i = AX_i + B, \quad Z_i = CX_i + D,
\]
and $Z_i = X_i^TY_i - \beta_ic_iJ$, where $c_i = f(X_i)$.
\end{Cond}
In subsection \eqref{COND}, we will give an explicit example of values such that the Conditions \ref{c:1}-\ref{c:2}-\ref{c:3} are fulfilled.
\\
\\
Once this is achieved, we need to extend the energy $\E_f$ to an energy defined on integral currents of dimension $2$ in $\R^4$. Some of the results we present in this section in our specific case can be easily generalized to more general polyconvex integrands. Therefore, we defer their proofs to Section \ref{genext}. 
\\
\\
In order to extend our polyconvex function $f$ to a \emph{geometric} functional, we first recall \eqref{H}, i.e. $$f(X) = \eps\mathcal{A}(X) + g(X,\det(X)),$$ for $g: \R^{5} \to \R$ convex and smooth, and introduce the convex function $h: \R^5 \to \R$:
\[
h(z)\doteq \varepsilon\sqrt{1 + \|z\|^2} + g(z).
\] 
We consider the \emph{perspective function} of $h$:
\begin{equation}\label{G}
G(z,t)\doteq y h\left(\frac{z}{y}\right), \quad \forall z \in \R^5, y >0.
\end{equation}
It is a standard result in convex analysis that $G$ is convex on $\R^5\times \R_+$ as soon as $h$ is convex on $\R^5$, compare \cite[Lemma 2]{DBM}. Property \eqref{infty} reads as
\begin{equation}\label{outball}
h(z) = (M+\eps)\sqrt{1 + \|z\|^2} - L,\quad \forall z\in B_R^c(0),
\end{equation}
therefore we also find that the \emph{recession} function of $G$ is 
\[
h^*(z) = \lim_{y \to 0^+}G(z,t) = M\|z\|,\quad \forall z \in \R^5.
\]
Hence, $G$ can be extended to the hyperplane $y = 0$ as
\[
G(z,0) \doteq h^*(z).
\]
In Lemma \ref{l:2}, we will prove that $G(z,t)$ admits a finite, positively 1-homogeneous convex extension $\mathcal{G}$ to the whole space $\R^6$. We are finally able to define an integrand on the space of 2-vectors of $\R^4$, $\Lambda_2(\R^4)$. For a more thorough introduction to $k$-vectors, see Subsection \ref{multa}. Recall that
\[
\Lambda_2(\R^4) = \spn\{v_1\wedge v_2: v_1,v_2 \in \R^4\}.
\]
A basis for $\Lambda_2(\R^4)$ is given by the six elements $e_i\wedge e_j, 1\le i < j\le 4$, where $e_1,e_2,e_3,e_4$ is the canonical basis of $\R^4$. Recall moreover that this vector space can be endowed with a scalar product that acts on simple vectors as
\[
\langle v_1\wedge v_2, w_1\wedge w_2\rangle\doteq \det\left(\begin{array}{cc}
(v_1,w_1) & (v_1,w_2)\\
(v_2,w_1) & (v_2,w_2)
\end{array}\right),
\]
where $(u,v)$ denotes as usual the standard scalar product of $\R^4$. The integrand
\[
\Psi: \Lambda_2(\R^4) \to \R,
\]
is thus defined as, for $\tau \in \Lambda_2(\R^4)$,
\begin{equation}\label{psi}
\Psi(\tau) \doteq \mathcal{G}(\langle \tau,e_3\wedge e_2\rangle,\langle \tau,e_4\wedge e_2\rangle,\langle \tau,e_1\wedge e_3\rangle,\langle \tau,e_1\wedge e_4\rangle,\langle \tau,e_3\wedge e_4\rangle,\langle \tau,e_1\wedge e_2\rangle).
\end{equation}
Consequently, we define an energy on $\mathcal{I}_2(\R^4)$ as
\[
\Sigma(T) \doteq \int_E\Psi(\vec T(z))\theta(z) d\mathcal{H}^2(z),
\]
if $T = \llbracket E, \vec T, \theta\rrbracket$. For the notation concerning rectifiable currents and graphs, we refer the reader to Subsection \ref{currents}. The energy defined in this way satisfies Almgren's ellipticity condition \eqref{UALM}, as we will prove in Lemma \ref{l:3}. Finally, in Lemma \ref{l:4}, we will prove that the current 
\begin{equation}\label{Tutheta}
T_{u,\theta} = \llbracket\Gamma_u,\vec\xi_u,\theta\rrbracket
\end{equation}
is stationary for the energy $\Sigma$. The definition of stationarity for geometric functionals is recalled in Section \ref{GM}. In \eqref{Tutheta}, $\Gamma_u$ is the graph of $u$, $\vec\xi_u$ is its orientation, see \eqref{orientinggraph}, and $\theta(y)$ is a multiplicity, defined as $\theta(x,u(x)) = \beta_i$ if $x \in \Omega$ is such that
\[
Dw(x) = \left(
\begin{array}{cc}
X_i\\
Y_i\\
Z_i\\
\end{array}
\right) = \left(\begin{array}{c}X_i\\ \beta_iDf(X_i)\\ \beta_iX_i^TDf(X_i)J - \beta_if(X_i)J\end{array}\right).
\]
This discussion constitutes the proof of the following:

\begin{theorem}
There exists a smooth and elliptic integrand $\Psi: \Lambda_2(\R^4)\to \R$ such that the associated energy $\Sigma$ admits a stationary point $T$ whose (integer) multiplicities are not constant. Moreover the rectifiable set supporting $T$ is given by a graph of a Lipschitz map $u: \Omega \to \R^2$ that fails to be $C^1$ in any open subset $\mathcal{V} \subset \Omega$.
\end{theorem}

\begin{lemma}\label{l:1}
There exists a smooth function $f: \R^{2\times 2}\to \R$ of the form
\[
f(X)\doteq \eps\mathcal{A}(X) + g(X,\det(X))
\]
with $g:\R^5 \to \R$ convex and smooth, such that
\begin{enumerate}
\item \eqref{prime} is fulfilled;
\item $g(X) = M\mathcal{A}(X) - L$ for constants $M,L > 0$, if $\|X\| \ge R$.
\end{enumerate}
\end{lemma}

\begin{proof}
We will follow roughly the strategy of \cite[Lemma 3]{LSP}. At first we construct the function $g$ in several steps. Let $\{(X_i, Y_i, Z_i, \beta_i)\}_{i=1}^5$ the set of admissible matrices. For $\varepsilon>0$ consider for each $i$ the perturbed values 
\begin{equation}
\begin{split}\label{perturval}
Y_i^\eps&\doteq Y_iJ - \varepsilon \beta_iD\mathcal{A}(X_i)J \\
	c_i^\varepsilon &\doteq c_i - \varepsilon \mathcal{A}(X_i)\\
	d_i^\varepsilon&\doteq d_i - \partial_ya(X_i, \det(X_i))
\end{split}
\end{equation}
where $a(X,d) = \sqrt{1+ |X|^2 + d^2}$ and $\mathcal{A}(X)= a(X,\det(X))$, as defined in \eqref{area}. 
Furthermore we introduce the perturbed matrix 
\[ Q_{ij}^\varepsilon\doteq c^\varepsilon_i - c^\varepsilon_j + d^\varepsilon_i\det(X_i - X_j) + \frac{1}{\beta_i}\langle X_i - X_j ,Y^\varepsilon_iJ\rangle. \]
Thanks to the strict inequality in \eqref{prime-1} we can fix $\varepsilon, \sigma>0$ such that $Q^\varepsilon_{ij}\le -\sigma <0$ for all $i,j$. Let us define the linear functions 
\[l_i(X,d)\doteq c^\eps_i - \frac{1}{\beta_i}\langle  Y^\eps_iJ, X-X_i\rangle + d_i \left( \langle\cof(X_i)^T, X_i - X\rangle + d - \det(X_i)\right)\,\]
and the convex function 
\[
g_1(X,d)\doteq \max_{1\le i \le 5}\, l_i(X,d).
\]
Note that $l_j(X_j,\det(X_j)) = c_j^\eps$ and 
\[ l_i(X_j, \det(X_j))= c^\eps_j + Q^\eps_{ij} < c^\eps_j\,. \]
Hence there is $\delta>0$ such that $l_i(X,d)< l_j(X,d)$ for all $(X,d) \in B_{\delta}(X_j,\det(X_j))$ for all $i\neq j$ which implies that $g_1 = l_j$ on $B_{\delta}(X_j,\det(X_j))$. Choosing a radial symmetric, non-negative smoothing kernel on $\R^5$, $\rho_\varepsilon$, $0< \varepsilon << \delta$ we have that $g_2\doteq \rho_\eps \star g_1$ satisfies
\begin{enumerate}
	\item $g_2$ is smooth and convex
	\item $g_2 = l_j$ in a neighbourhood of $(X_j, \det(X_j))$ for all $j \in \{1,\dots, 5\}$. 
	\item $|g_2(X,d)| \le C \|(1,X,d)\|$ for all $(X,d)$ for some $C>0$.
\end{enumerate}
We choose any $R> 2 \max_{1\le i \le 5} \{\|X_i\| + |\det(X_i)| \}$, and any $M>C$. Now we may choose $L>0$ such that 
\begin{equation}\label{minore}
F(X,d)\doteq M a(X,d) - L < g_2(X,d) \text{ on } B_R.
\end{equation}
Since $M>C$ we have that 
\begin{equation}\label{maggiore}
F(X,d) = M\|(1,X,d)\| - L > g_2(X,d)
\end{equation}
 for all $(X,d) \notin B_{R_2}$, for some $R_2 > R$. Now let us fix a smooth approximation of the $\max$ function, say 
\[
m(a,b)\doteq (\phi_\varepsilon\star\max)(a,b),
\] 
where $\phi_\eps$ is a radial symmetric, non-negative smoothing kernel in $\R^2$. Note that $m(a,b) = \max(a,b)$ outside a neighbourhood of $\{ a= b\}$. In particular if we choose $\varepsilon$ sufficiently small we can ensure that
\[g(X,d)\doteq m( F(X,d), g_2(X,d))\]
agrees with $g_2$ on $B_{\frac{2R}{3}}$ by \eqref{minore}, and that it agrees with $F(X,d)$ outside $B_{2R_2}$ by \eqref{maggiore}. 
It remains to check that $g(X,d)$ is still convex. 
First note that $\partial_am \ge 0$ and $\partial_bm\ge 0$ since $\partial_a\max = \mathbf{1}_{\{a>b\}} \ge 0$, $\partial_b\max =  \mathbf{1}_{\{b>a\}} \ge 0$. Now it is a direct computation on the Hessian to see that if $f_1, f_2 \in C^2(\R^N)$ are two convex functions and $\tilde{m} \in C^2(\R^2)$ is convex  with $\partial_a\tilde{m}(a,b), \partial_b\tilde{m}(a,b) \ge 0$, then the composition $k(x)\doteq \tilde{m}(f_1(x), f_2(x))$ is convex. Thus we conclude that $g$ is convex.
Let us summarize the properties of $g$ and the related polyconvex integrand $f_1(X)\doteq g(X, \det(X))$
\begin{enumerate}
	\item $g$ is a smooth, convex function;
	\item $g=M\, a - L$ outside a ball $B_{R_3}$
	\item $g = g_2$ on a ball $B_{R_0}$, that implies that $f_1(X_i)=c_i^\varepsilon$ and $\beta_i Df_1(X_i)J = Y_i^\varepsilon$ for all $i$. 
\end{enumerate}
In particular from the last conditions and \eqref{perturval} we conclude that $h(X,d)\doteq \varepsilon a(X,d) + g(X,d)$ is convex, $f(X)\doteq\varepsilon \mathcal{A}(X) + f_1(X)$ is smooth, polyconvex and satisfies the desired properties, in particular $f(X_i)=c_i$, $\beta_i Df(X_i)J= Y_i$ for all $i$ and $f = (\eps + M)\mathcal{A} - L$ outside a ball centered at $0$.
\end{proof}

\begin{lemma}\label{in-app}
Given a large $T_5$ configuration $\{X_1,\dots, X_5\} \subset \Sym(2)$, where $\Sym(2)$ is the space of symmetric matrices of $\R^{2\times 2}$, there exists a map $u \in \Lip(\Omega,\R^2)$ such that
\begin{equation}\label{bel}
Du \in \{X_1,\dots, X_5\}
\end{equation}
and such that for every open $\mathcal{V}\subset \Omega$,
\begin{equation}\label{osc}
|\{x \in \Omega: Du(x) = X_i\}\cap \mathcal{V}| > 0,\quad \forall i\in\{1,\dots, 5\}.
\end{equation}
\end{lemma}
\begin{proof}
This statement is well-known, so we will only sketch its proof and give references where to find the relevant results. As shown in \cite[Theorem 2.8]{FS}, if $K\doteq \{X_1,\dots, X_5\}$ forms a large $T_5$ configuration, then there exists an in-approximation of $K$ inside $\Sym(2)$. This means, compare \cite[Definition 1.3]{FS}, that there exists a sequence of sets $\{U_k\}_{k \in \N}$, open relatively to $\Sym(2)$, such that
\begin{itemize}
\item $\sup_{X \in U_k}\dist(X,K) \to 0$ as $k \to \infty$;
\item $U_k\subset U_{k + 1}^{rc}, \forall k \in \N$.
\end{itemize}
For a compact $C \subset \R^{2\times 2}$, the rank-one convex hull is defined as
\[
C^{rc} \doteq \{P \in \R^{2\times 2}: f(P) \le 0, \forall f \text{ rank-one convex such that } \sup_{X \in C}f(X) \le 0\},
\]
where $f: \R^{2\times 2} \to \R$ is said to be rank-one convex if
\[
f(tA + (1-t)B) \le tf(A) + (1-t)f(B),\quad \forall A,B \in \R^{2\times 2}, \det(A - B) = 0, t \in[0,1].
\]
For an open set $U \subset \R^{2\times 2}$, 
\[
U^{rc} \doteq \bigcup_{C \subset U, C \text{ compact }}C^{rc}.
\]
In this way, if $U$ is open, then $U^{rc}$ is open as well. The existence of a in-approximation for $K$ implies the existence of a non-affine map $u$ such that $Du \in \{X_1,\dots, X_5\}$, hence $\eqref{bel}$. This is proved in \cite[Theorem 1.1]{FS}. To show \eqref{osc}, there are two ways. Either, one can use the same proof of \cite[Theorem 4.1]{SMVS} or \cite[Proposition 2]{LSP} to show that the essential oscillation of $Du$ is positive on any open subset of $\Omega$. Since there is rigidity for the four gradient problem, see \cite{CMKB}, this implies \eqref{osc}. Another way to show \eqref{osc} is to use the Baire Theorem approach of convex integration as introduced by Kirchheim in \cite{KIRK}. In particular, in \cite[Corollary 4.15]{KIRK}, it is proved the following. Define
\[
\mathcal{U}\doteq \bigcup_{k \in \N}U^{rc}_k,
\]
we fix $A \in \mathcal{U}$, and we also set
\[
\mathcal{P} \doteq \{v \in \Lip(\Omega,\R^2): Dv \in \mathcal{U}, v \text{ piecewise affine}, v|_{\partial \Omega} = A\},
\]
then the typical (in the sense of Baire) map 
\[
u \in \mathcal{P}^{\|\cdot\|_\infty}
\]
has the property that $Du \in K$. Then, we can use \cite[Lemma 7.4]{TR} to show that actually the typical map is non-affine on any open set, hence again by the rigidity for the four gradient problem, we conclude \eqref{osc}.
\end{proof}

\begin{lemma}\label{l:2}
Let $G: \R^{5}\times \R_{\ge 0}\to\R$ be the convex function defined in $\eqref{G}$. Then, there exists a positively $1$-homogeneous, convex function $\mathcal{G} \in C^\infty(\R^6\setminus\{0\})\cap \Lip(\R^6)$ such that
\[
\mathcal{G}(z,t) = G(z,t),
\]
if $z \in \R^5,t \in \R_+$.
\end{lemma}

\begin{proof}
To prove the statement, it is sufficent to notice that the convexity of $h$ and \eqref{outball} tells us that $h$ has property (P), see the beginning of Section \ref{genext}, and therefore we can simply apply Proposition \ref{genextprop}. The smoothness is a consequence of the smoothness of $h$, property \eqref{outball} and Corollary \ref{repres}.
\end{proof}

\begin{lemma}\label{l:3}
The energy $\Sigma_\Psi$ satisfies the uniform Almgren ellipticity condition \eqref{UALM}.
\end{lemma}
\begin{proof}
By construction, it is immediate to see that also
\[
\mathcal{G}_\eps(z,t)\doteq \mathcal{G}(z,t) - \frac{\eps}{2}\sqrt{t^2 + \|z\|^2}
\]
is still convex and positively $1$-homogenous. Define $\Psi_\eps$ as in \eqref{psi} by substituting $\mathcal{G}_\eps$ to $\mathcal{G}$. By the general Proposition \ref{Alm}, we see that $\Sigma_{\Psi_\eps}$ satisfies Almgren condition, hence $\Sigma_\Psi$ satisfies \eqref{UALM} with constant $\frac{\eps}{2}$.
\end{proof}

\begin{lemma}\label{l:4}
The current $T_{u,\theta} = \llbracket\Gamma_u,\vec\xi_u,\theta\rrbracket$ defined in \eqref{Tutheta} is stationary in $\Omega\times \R^2$ for the energy $\Sigma_\Psi$.
\end{lemma}
\begin{proof}
A direct computation shows that $f$ and $\Psi$ fulfill
\[
f(X) = \Psi(W(X))\mathcal{A}(X), \forall X \in \R^{2\times 2},
\]
where $W(X) = M^1(X)\wedge M^2(X)$ and $M^i$ are the columns of the matrix
\[
M(X)\doteq
\left(
\begin{array}{c}
\id_m\\
X
\end{array}
\right).
\]
Once this is checked, the proof is entirely analogous to the one of \cite[Proposition 6.8]{DLDPKT}, and will be sketched in the appendix, see Proposition \ref{vargen}.
\end{proof}

\subsection{Explicit values}\label{COND}
Define the following quantities:
\begin{align*}
&(\beta_1,\beta_2,\beta_3,\beta_4,\beta_5) \doteq (2,5,10,1,2);\\
&(d_1,d_2,d_3,d_4,d_5) \doteq  \left(-\frac{1204}{828115},0,\frac{-1309}{454800},\frac{-10097}{2546880},0\right);\\
&(c_1,c_2,c_3,c_4,c_5) \doteq \left(0,0,-\frac{2929}{1137000},\frac{5233}{113700}, -\frac{33}{15160}\right).
\end{align*}
The large $T_5$ configuration is given by:
\[
A_1\doteq
\left(
\begin{array}{cc}
\frac{8}{5}& -2\\[6pt]
-2 & \frac{8}{5}\\[6pt]
 -\frac{8}{1137} & \frac{7361}{454800}\\[6pt]
\frac{267}{151600}& \frac{8}{1137}\\[6pt]
\frac{-3361}{227400}& \frac{3361}{284250}\\[6pt]
\frac{4801}{284250} & -\frac{4801}{227400}
\end{array}
\right);
A_2\doteq
\left(
\begin{array}{cc}
\frac{8}{5}& 2\\[6pt]
2 & \frac{8}{5}\\[6pt]
 \frac{8}{1137} & \frac{7361}{454800}\\[6pt]
\frac{267}{151600}& -\frac{8}{1137}\\[6pt]
\frac{3361}{227400}& \frac{3361}{284250}\\[6pt]
\frac{4801}{284250} & \frac{4801}{227400}
\end{array}
\right);
A_3\doteq
\left(
\begin{array}{cc}
\frac{2}{5}& 0\\[6pt]
0 & -\frac{18}{5}\\[6pt]
 0 & -\frac{959}{454800}\\[6pt]
\frac{907}{151600}& 0\\[6pt]
0& -\frac{10083}{379000}\\[6pt]
\frac{4801}{1137000} & 0
\end{array}
\right);
\]

\[
A_4\doteq
\left(
\begin{array}{cc}
-\frac{18}{5}& 0\\[6pt]
0 & \frac{2}{5}\\[6pt]
 0 & \frac{5441}{454800}\\[6pt]
\frac{9121}{454800}& 0\\[6pt]
0& \frac{3361}{1137000}\\[6pt]
-\frac{14403}{379000} & 0
\end{array}
\right);
A_5\doteq
\left(
\begin{array}{cc}
\frac{3}{4}& 0\\[6pt]
0 & \frac{3}{4}\\[6pt]
0 & \frac{6001}{454800}\\[6pt]
\frac{2161}{454800}& 0\\[6pt]
0& \frac{3361}{606400}\\[6pt]
\frac{4801}{606400} & 0
\end{array}
\right).
\]
Define $X_i, Y_i,Z_i \in \R^{2\times 2}$ through the relations
\[
\left(
\begin{array}{cc}
X_i\\
Y_i\\
Z_i\\
\end{array}
\right)= A_i.
\]
The matrices $A,B,C,D$ appearing in Condition \ref{c:3} are given by:
\[
A\doteq
\left(
\begin{array}{cc}
0&  \frac{4}{1137}\\
- \frac{4}{1137} & 0
\end{array}
\right);
B\doteq
\left(
\begin{array}{cc}
0&  \frac{4801}{454800}\\
\frac{3361}{454800} & 0
\end{array}
\right);
C\doteq
\left(
\begin{array}{cc}
0&  \frac{3361}{454800}\\
\frac{4801}{454800} & 0
\end{array}
\right), D\doteq 0.
\]
These values fulfill Conditions \ref{c:1}, \ref{c:2}, \ref{c:3}. In particular, the three permutations in the definition of large $T_5$ configuration of Condition \ref{c:2} are: $[1,2,3,5,4]$, $[1,2,4,5,3]$, $[1,2,5,3,4]$.

\section{Extension of polyconvex functions}\label{genext}

Let $\Phi: \R^{n\times m} \to \R^{k}$ be the usual map that, to a matrix $X \in \R^{n\times m}$, associates the vector of the subdeterminants of $\Phi$. Consider a polyconvex function
\[
f(X) = h(\Phi(X)),
\]
$h: \R^k \to \R$ being\footnote{This hypothesis on the regularity of $h$ is not necessary, and one could simply consider $h \in \Lip(\R^k)$. Indeed, all the results of this section would work with simple modifications in the Lipschitz case. Nonetheless, we prefer to assume $C^1$ regularity in order to avoid further technicalities.} $C^1$. The purpose of this section is to generalize the arguments of the previous section to arbitrary $n,m$, and hence to prove some of the lemmas of that section. Consider the following set of assumptions
\begin{enumerate}[(i)]
\item\label{con} $h$ is convex;
\item\label{lin} $h$ has linear growth, i.e. $|h(z)| \le A\|z\| + B, \forall z \in \R^k$, for $A,B\ge 0$;
\item\label{lim} $\lambda \doteq \inf \{h(z) - (Dh(z),z): z\in \R^k\} > -\infty$;
\item\label{magg} $(Dh(z_2),z_2 - z_1) \le h(z_1) + h(z_2), \quad \forall z_1,z_2 \in \R^k$.
\end{enumerate}
If $h$ fulfills \eqref{con}-\eqref{lin}-\eqref{lim}, we will say it has property (P). If, in addition, $h$ satisfies \eqref{magg}, we will say that $h$ fulfills property (PE).
\begin{remark}\label{red}
Notice that \eqref{lim} is a consequence of \eqref{magg}, indeed if \eqref{magg} holds we can write, for $z_1 = 0$ and for any $z_2 = z \in \R^k$:
\[
(Dh(z),z) \le h(0) + h(z),
\]
hence
\[
-h(0) \le h(z) - (Dh(z),z), \quad \forall z \in \R^k,
\]
that implies \eqref{lim}.
\end{remark}
\noindent We denote with $h^*$ the \emph{recession function} of $h$:
\[
h^*(x)\doteq \lim_{y\to 0^+}yh\left(\frac{x}{y}\right), \quad \forall x \in \R^k.
\]
It is not difficult to prove that the limit above always exists and is finite for a function $h$ satisfying (P). To show it, one can use the fact that the function $$y \mapsto yh\left(\frac{x}{y}\right)$$ defined for $y > 0$ is convex for every fixed $x \in \R^k$, see \cite[Lemma 2]{DBM}.
\\
\\
As above, we define the \emph{perspective function}
\[
G(x,y)\doteq yh\left(\frac{x}{y}\right),\quad \text{if } y > 0.
\]
We consider the smallest convex extension of $G$ to the whole $\R^{k + 1}$:
\[
\mathcal{G}(z,t) \doteq \sup\{G(x,y) + (DG(x,y),(z,t) - (x,y)): (x,y) \in \R^{k}\times (0,+\infty)\}.
\]
By $1$-homogeneity of $G$,  we can write
\begin{equation}\label{GG}
\mathcal{G}(z,t) = \sup\{(DG(x,y),(z,t)): (x,y) \in \R^{k}\times (0,+\infty)\}
\end{equation}
First, we prove
\begin{prop}\label{genextprop}
Let $\mathcal{G}$ be defined as in \eqref{GG}. Then, if $h$ satisfies $(P)$, $\mathcal{G}$
\begin{enumerate}
\item\label{co} is convex and extends $G$ on $\R^k\times (0,+\infty)$;
\item\label{po} is positively 1-homogeneous;
\item\label{fo} is finite everywhere.
\end{enumerate}
Conversely, if there exists a function $\mathcal{G}$ that fulfills \eqref{co}-\eqref{po}-\eqref{fo}, then $h$ fulfills (P).
\end{prop}

Furthermore, we can prove the following characterization of $\mathcal{G}$:

\begin{corollary}\label{repres}
Let $h$ fulfill property (P), and let $\mathcal{G}$ be defined as in \eqref{GG}. Assume further that there exists $\lambda' \in \R$ and $R > 0$ such that 
\begin{equation}\label{recball}
h(z) = h^*(z) + \lambda',\quad \text{ for } \|z\|\ge R.
\end{equation}
 Then, $\lambda' = \lambda$ and for $t < 0$, we have
\[
\mathcal{G}(z,t) = h^*(z) + \lambda' t,
\]
where $\lambda$ is the quantity appearing in \eqref{lim}.
\end{corollary}

Before starting with the proof of the proposition, we need to recall some results concerning the notion of subdifferential at $x \in \R^N$ of a convex function $f: \R^{N} \to \R$.

\subsection{Subdifferentials}\label{SUB} The subdifferential of $f$ at $x$, denoted with $\partial f(x)$, is the collection of those vectors $v \in \R^N$ such that
\[
(v,y-x) \le f(y) - f(x), \quad \forall y \in \R^N.
\]
We will use the following facts concerning the subdifferential. For a convex function with finite values, $\partial f(x) \neq \emptyset$ at all $x \in \R^N$, see \cite[Theorem 23.4]{ROCK}. Conversely, if $f:\R^N \to \R$ is such that $\partial f(x) \neq \emptyset$ at every $x \in \R^N$, then $f$ is convex, since in that case
\[
f(x) = \sup_{y \in \R^N}\sup_{v \in \partial f(y)}\{(v,x - y) + f(y)\}.
\]
As can be seen from the definition of subdifferential,
\[
|f(x) - f(y)|\le \max\left\{\sup_{v \in \partial f(x)}\|v\|,\sup_{w \in \partial f(y)}\|w\|\right\}\|x - y\|.
\]
This, together with the fact that if $K$ is compact, then $\partial f(K) \doteq \bigcup_{x \in K}\partial f(x)$ is compact, see \cite[Lemma A.22]{FIG}, yields the fact that every convex function is locally Lipschitz. Moreover, if $f$ is positively $1$-homogeneous, a simple application of the definition of subdifferential shows that
\begin{equation}\label{lambdasymm}
v \in \partial f(x) \Leftrightarrow v \in \partial f(\lambda x), \quad \forall \lambda > 0, x \in \R^N.
\end{equation}
In particular, combining \eqref{lambdasymm} with the local Lipschitz property of convex functions, we infer that if $f: \R^N \to \R$ is convex and positively-1 homogeneous, $f$ must be globally Lipschitz. Furthermore, using the definition of subdifferential and \eqref{lambdasymm} for $f$ convex and positively-1 homogeneous, it is easy to see that the following \emph{generalized Euler's formula} holds
\begin{equation}\label{euler}
(v,x) = f(x), \quad \forall v \in \partial f(x), \forall x \in \R^N.
\end{equation}
Finally, we recall that at $x$, the convex function $f: \R^N \to \R$ is differentiable if and only if
\[
\partial f(x) = \{Df(x)\},
\]
see \cite[Lemma A.20-A.21]{FIG} and references therein. We can now start the proof of the proposition.

\subsection{Proof of Proposition \ref{genextprop}}
First we assume that $h$ has property (P). $\mathcal{G}$ is convex since it is supremum of linear functions. Moreover, the convexity of $h$ yields the convexity of $G$ on $\R^k\times (0,+\infty)$. Having established that $G$ is convex, the fact that $\mathcal{G}$ as in \eqref{GG} extends $G$ is a classical fact. This proves \eqref{co}. Since $G$ was positively $1$-homogeneous, we have that $\mathcal{G}$ is as well homogeneous. Therefore \eqref{po} is checked, and we only need to prove \eqref{fo}. By \eqref{GG} we see that in order to conclude we only need to show that, for fixed $(z,t) \in \R^{k + 1}$,
\[
(DG(x,y),(z,t)) \le L < + \infty, \quad \forall (x,y) \in \R^k\times (0,+\infty),
\]
where $L$ possibly depends on $(z,t)$. Let us compute $DG$. Firstly we have that $$\partial_{x_i}G(x,y) = \partial_{x_i}h\left(\frac{x}{y}\right).$$ Now, exploiting the convexity of $h$, we can choose any $v \in \R^k$ with $\|v\|= 1$ and write
\[
(Dh(x),v) \le \frac{h(x + sv) - h(x)}{s}, \quad \forall s \in \R^+.
\]
Using the linear growth of $h$, i.e. \eqref{lin}, we bound:
\[
(Dh(x),v) \le \frac{A\|x + sv\| + B + A\|x\| + B}{s}.
\]
Letting $s \to + \infty$, the previous expression yields
\begin{equation}\label{equibound}
(Dh(x),v) \le A,\quad \forall x,v \in \R^n, \|v\|=1.
\end{equation}
Thus, if we can show that $\partial_yG(x,y)$ is uniformly bounded, then we conclude the proof. We compute explicitly, for every $(x,y)\in \R^{k}\times (0,+\infty)$
\[
\partial_yG(x,y) = h\left(\frac{x}{y}\right)  - \left(Dh\left(\frac{x}{y}\right),\frac{x}{y}\right).
\]
We are therefore left to study the boundedness (from below) of the function $z\mapsto h(z) - (Dh(z),z)$, but this is a consequence of \eqref{lim} of property (P).
\\
\\
Finally, let us show the necessity of (P). If $\mathcal{G}$ is convex and extends $G$, then in particular
\[
\mathcal{G}(z,1) = G(z,1) = h(z), \quad \forall z \in \R^k,
\]
hence $h$ is convex. By the discussion of Subsection \ref{SUB}, we know that $\mathcal{G}$ is globally Lipschitz with constant $L > 0$. Since
\[
\mathcal{G}(z,1) = h(z), \quad \forall z \in \R^k,
\] 
we infer that $h$ has linear growth, i.e. it enjoys property \eqref{lin}. Finally, we need to show \eqref{lim}. Since $\mathcal{G}$ extends $G$ in the upper half-space, we obtain
\[
|\partial_yG(x,y)| \le L, \quad \forall (x,y) \in \R^k\times(0,+\infty).
\]
By the definition of $G$, we deduce
\[
|\partial_yG(z,1)| = |h(z) - (Dh(z),z)| \le L, \quad \forall z \in \R^k,
\]
hence \eqref{lim}.

\subsection{Proof of Corollary \ref{repres}}

First we show that $\lambda = \lambda'$. To see this, consider for any $z \neq 0$ the auxiliary function $g(t) \doteq h(tz) - (Dh(tz),tz)$, for $t > 0$. Then, $g$ is non-increasing. Indeed,
\[
g\left(\frac{1}{t}\right) = \partial_y\mathcal{G}(z,t),
\]
and we can use that $\mathcal{G}$ is convex to deduce that $t\mapsto \partial_y\mathcal{G}(z,t)$ is non-decreasing, hence that $t\mapsto g(t)$ is non-increasing. Now, for any $t$ sufficiently large, by assumption \eqref{recball}, we have that
\[
h(tz) - (Dh(tz),tz) = \lambda'.
\]
This shows that
\[
\lambda' = \lim_{t \to +\infty}[h(tz) - (Dh(tz),tz)] = \inf_{t > 0}[h(tz) - (Dh(tz),tz)] \ge \lambda.
\]
In particular, notice that $h(0) = \lim_{t \to 0^+}g(t) \ge \lambda'$. To show the equality between $\lambda$ and $\lambda'$, consider now a sequence $z_n \in \R^k$ such that $a_n \doteq h(z_n) - (Dh(z_n),z_n) \to \lambda$ as $n \to \infty$. If $z_n = 0$ for infinitely many $n$, we can write $\lambda = \lim_{n \to \infty}a_n = h(0) \ge \lambda'$ and the proof is concluded. Otherwise, by the computation above, we have, for every $t \ge 1$
\[
a_n \ge h(tz_n) - (Dh(tz_n),tz_n).
\]
By choosing $t$ in dependence of $z_n$, we can ensure through the assumption \eqref{recball} that
\[
h(tz_n) - (Dh(tz_n),tz_n) = \lambda'.
\]
Therefore,
\[
\lambda = \lim_na_n \ge \lambda'
\]
and the proof of the first part of the Corollary is finished.
\\
\\
Now we wish to show the characterization of $\mathcal{G}$. Fix $(z,t) \in \R^k\times (-\infty,0)$. Let $(x,y) \in \R^k\times (0,+\infty)$. Then, using the definition $G(x,y) = yh\left(\frac{x}{y}\right)$
\begin{equation}\label{expcom}
(DG(x,y),(z,t)) = \left(Dh\left(\frac{x}{y}\right),z\right) + \left(h\left(\frac{x}{y}\right) - \left(Dh\left(\frac{x}{y}\right),\frac{x}{y}\right)\right)t.
\end{equation}
By \eqref{lim}, we get
\[
h\left(\frac{x}{y}\right) - \left(Dh\left(\frac{x}{y}\right),\frac{x}{y}\right) \ge \lambda,
\]
hence, since $t < 0$, then
\[
(D\mathcal{G}(x,y),(z,t)) \le \left(Dh\left(\frac{x}{y}\right),z\right) + \lambda t.
\]
We now show that
\begin{equation}\label{secondpart}
\left(Dh\left(\frac{x}{y}\right),z\right) \le h^*(z).
\end{equation}
Let $a,b \in \R^k$, $r > 0$. Then, using the convexity of $h$,
\[
0 \le \left(Dh(a) - Dh\left(\frac{b}{r}\right), a - \frac{b}{r}\right),
\]
or
\begin{equation}\label{eqr}
0 \le \left(Dh(a) - Dh\left(\frac{b}{r}\right), ra - b\right).
\end{equation}
To conclude \eqref{secondpart}, we might use assumption \eqref{recball}, but let us use a slightly more general argument in order to use the same inequality below. By \eqref{equibound}, we have that $$\left\{Dh\left(\frac{b}{r}\right)\right\}_{r > 0}$$ is an equibounded family of vectors, hence up to subsequences it admits a limit $\lim_{j \to \infty} Dh\left(\frac{b}{r_j}\right) = w \in \R^k$, where $\lim_{j \to \infty}r_j = 0$. Hence,
\begin{equation}\label{interineq}
0 \le \lim_{j \to \infty}\left(Dh(a) - Dh\left(\frac{b}{r_j}\right), r_ja - b\right) =  -(Dh(a) - w,b).
\end{equation}
Now, $w \in \partial h^*(b)$, in fact using the convexity of $h$ we can write
\[
\left(Dh\left(\frac{b}{r_j}\right),\frac{a}{r_j}-\frac{b}{r_j}\right) \le h\left(\frac{a}{r_j}\right) - h\left(\frac{b}{r_j}\right).
\]
Multiplying by $r_j$ and letting $j \to \infty$, we find that $w \in \partial h^*(b)$. By \eqref{euler}, \eqref{secondpart} now follows from \eqref{interineq}. Therefore, we can conclude that, for $y < 0$,
\[
\mathcal{G}(z,t) \le h^*(z) +\lambda t.
\]
To conclude the assertion, we consider for any $y > 0$:
\[
(D\mathcal{G}(z,y),(z,t)) = \left(Dh\left(\frac{z}{y}\right),z\right) + \left(h\left(\frac{z}{y}\right) - \left(Dh\left(\frac{z}{y}\right),\frac{z}{y}\right)\right)t.
\]
If we choose $y$ sufficiently small (in dependence of $z$), once again using \eqref{recball}, we see that
\[
(D\mathcal{G}(z,y),(z,t)) = h^*(z) + \lambda' t = h^*(z) + \lambda t,
\]
the latter being true by the first part of the proof. This concludes the proof of the corollary.

\subsection{Symmetric Extension}

Now we show the link between (PE) and a symmetric extension. Notice that imposing that $h$ admits a $1$-homogeneous and even extension such that $\mathcal{G}(z,1) = h(z)$ forces this extension to have the form
\begin{equation}\label{nec}
\mathcal{G}(z,t) = |t|h\left(\frac{z}{t}\right)
\end{equation}
for $t \neq 0$. If we require that $\mathcal{G}$ is convex too, then it is continuous, hence it becomes uniquely determined on $\{(x,y): y  =  0\}$ as $\mathcal{G}(z,0) = h^*(z)$. Therefore, instead of considering a general convex extension as in \eqref{GG}, we are going to work with the function $\mathcal{G}$ obtained in \eqref{nec}.

\begin{prop}\label{genextpropsym}
$h$ satisfies (PE) if and only if $\mathcal{G}: \R^{k + 1}\to \R$ defined as
\begin{equation}
\mathcal{G}(z,t) =
\begin{cases}
 |y|h\left(\frac{z}{y}\right), &\text{ if } y \neq 0\\
h^*(z), &\text{ if } y = 0
\end{cases}
\end{equation}
is even and convex.
\end{prop}

\begin{proof}
Assume that $h$ satisfies (PE). First we prove that $\mathcal{G}$ is even. This amounts to show that
\begin{equation}\label{symm}
h^*(z) = \lim_{t\to 0^+}\frac{h(tz)}{t} = \lim_{t\to 0^+}\frac{h(-tz)}{t} = h^*(-z).
\end{equation}
To see this, we simply evaluate \eqref{magg} at $z_1 = \frac{-z}{t}$ and $z_2 = \frac{z}{t}$ for any $z \in \R^k$, $t > 0$ to find
\begin{equation}\label{2eq}
2\left(Dh\left(\frac{z}{t}\right),\frac{z}{t} \right) \le h\left(\frac{z}{t}\right) + h\left(-\frac{z}{t}\right).
\end{equation}
We now use the same argument to prove \eqref{secondpart} to see that for a sequence of positive numbers $\{t_j\}_{j \in \N}$ with $\lim_jt_j = 0$,  $\lim_{j \to \infty}Dh\left(\frac{z}{t_j}\right) = w \in \partial h^*(z)$. Therefore, multiplying by $t$ in \eqref{2eq} and passing to the limit along this subsequence, we get
\[
2(w,z) \le h^*(z) + h^*(-z),\quad \forall z \in \R^k.
\]
By \eqref{euler}, $(w,z) = h^*(z)$, and in this way we see that, using the last equation,
\[
2h^*(z) \le h^*(z) + h^*(-z) \Rightarrow h^*(z) \le h^*(-z),\quad \forall z \in \R^k,
\]
that implies \eqref{symm}.
\\
\\
Now we show that $\mathcal{G}$ is convex. We rely on the results of Subsection \ref{SUB}, and we aim to show that at every point $p = (z,t) \in \R^{k + 1}$,
\[
\partial \mathcal{G}(p) \neq \emptyset.
\]
Let first $t > 0$. Since at $p$ the function is differentiable, the only possible candidate for an element of the subdifferential is $v\doteq D\mathcal{G}(p)$. Notice moreover that by the 1-homogeneity of $\mathcal{G}$, $(D\mathcal{G}(p),p)=\mathcal{G}(p)$. Thus we have, for any $q = (x,y) \in \R^{k + 1}$:
\begin{equation}\label{eqhom}
(D\mathcal{G}(p),q-p) \le \mathcal{G}(q) -\mathcal{G}(p) \Leftrightarrow (D\mathcal{G}(p),q) \le \mathcal{G}(q).
\end{equation}
If we establish $(D\mathcal{G}(p),q) \le \mathcal{G}(q)$ for any $y \neq 0$, then we can use the pointwise convergence $$\lim_{y \to 0^+}\mathcal{G}(x,y) = \lim_{y \to 0^+}yh\left(\frac{x}{y}\right) = h^*(x) = \mathcal{G}(x,0)$$ to infer that the inequality holds also for $y = 0$. We therefore compute, for any $y \neq 0$:
\[
(D\mathcal{G}(p),q) = \left(Dh\left(\frac{z}{t}\right),x\right) + h\left(\frac{z}{t}\right)y - \left(Dh\left(\frac{z}{t}\right),\frac{z}{t}\right)y = y\left[h\left(\frac{z}{t}\right) - \left(Dh\left(\frac{z}{t}\right),\frac{z}{t}-\frac{x}{y}\right)\right]
\]
Using \eqref{eqhom}, $v = D\mathcal{G}(p)$ is a supporting hyperplane if and only if
\begin{equation}\label{teq}
(D\mathcal{G}(p), q) = y\left[h\left(\frac{z}{t}\right) - \left(Dh\left(\frac{z}{t}\right),\frac{z}{t}-\frac{x}{y}\right)\right] \le \mathcal{G}(q).
\end{equation}
Following the same argument of the beginning of Proposition \ref{genextprop}, since $h$ is convex, $\mathcal{G}$ is convex on $\R^k\times(0,+\infty)$. Thus \eqref{teq} is surely fulfilled if $y > 0$. If $y < 0$, \eqref{teq} becomes
\[
y\left[h\left(\frac{z}{t}\right)- \left(Dh\left(\frac{z}{t}\right),\frac{z}{t} - \frac{x}{y}\right)\right] \le  \mathcal{G}(q) = -yh\left(\frac{x}{y}\right),
\]
that can be rewritten as
\[
h\left(\frac{z}{t}\right)- \left(Dh\left(\frac{z}{t}\right),\frac{z}{t} - \frac{x}{y}\right) \ge -h\left(\frac{x}{y}\right), \forall (z,t) \in \R^k\times(0,+\infty),(x,y) \in \R^k\times (-\infty,0).
\]
The last condition is equivalent to \eqref{magg}. Now we need to prove that also for points $p = (z,t)$ with $t < 0 $ an element in the subdifferential exists. This is anyway a consequence of the evenness of $\mathcal{G}$ and the proof above, indeed the evenness of $\mathcal{G}$ yields
\[
D\mathcal{G}(p) = -D\mathcal{G}(-p),\quad \forall p \in \R^{k}\times (\R\setminus\{0\}).
\]
Therefore, for any $q = (x,y) \in \R^{k + 1}$,
\begin{align*}
(D\mathcal{G}(p),q - p) &= -(D\mathcal{G}(-p),q - p) =  (D\mathcal{G}(-p),p - q) = (D\mathcal{G}(-p),- q - (-p)) \\
&\le \mathcal{G}(-q)-\mathcal{G}(-p) = \mathcal{G}(q)-\mathcal{G}(p),
\end{align*}
where we exploited the fact that $D\mathcal{G}(-p) \in \partial\mathcal{G}(-p)$, as proved above. Finally, we need to produce an element in the subdifferential at points $p = (z,0)$. To do so, we again use the fact that for any $p' = (z,t)$ with $t > 0$, $q = (x,y) \in \R^{k + 1}$,
\[
(D\mathcal{G}(p'),q - p')\le \mathcal{G}(q) - \mathcal{G}(p').
\]
We only need to observe that $\{D\mathcal{G}(z,t)\}_{t > 0}$ is an equibounded family of vectors. This allows us to choose a sequence $t_j>0$ convergent to $0$ such that $\{D\mathcal{G}(z,t_j)\}_{j \in \N}$ converges to a vector $w \in \R^{k + 1}$. Since $\lim_{t\to 0^+}\mathcal{G}(z,t) = \mathcal{G}(z,0)$, we have
\[
(w,q - p) = \lim_{j \to \infty}(D\mathcal{G}(p_j),q - p_j) \le \lim_{j \to \infty}(\mathcal{G}(q) - \mathcal{G}(p_j)) = \mathcal{G}(q) - \mathcal{G}(p).
\]
where $p_j = (z,t_j), \forall j \in \N$. To show the equi-boundedness of $\{D\mathcal{G}(z,t)\}_{t>0}$, we observe that
\[
D_z\mathcal{G}(z,t) = Dh\left(\frac{z}{t}\right),
\]
that is equibounded in $z$ and $t$ by \eqref{equibound}. Exactly as in the proof of Proposition \ref{genextprop}, we use \eqref{lim} to say that
\[
\lambda \le \partial_t \mathcal{G}(z,t),\quad \forall (z,t) \in \R^k \times (0,+\infty).
\]
Hence we only need to provide a bound from above. To show it, we use the convexity of $h$ to estimate 
\[
\partial_t \mathcal{G}(z,t) = h\left(\frac{z}{t}\right) - \left(Dh\left(\frac{z}{t}\right),\frac{z}{t}\right) = h\left(\frac{z}{t}\right) + \left(Dh\left(\frac{z}{t}\right),0 - \frac{z}{t}\right) \le h\left(\frac{z}{t}\right)+ h(0)- h\left(\frac{z}{t}\right) = h(0),
\]
that provides the desired bound. This finishes the proof of the convexity of $\mathcal{G}$.
\\
\\
To conclude, we need to show the converse statement, i.e. that if $\mathcal{G}$ is even and convex, then $h$ fulfills (PE). The fact that $h$ fulfills \eqref{con}-\eqref{lin}-\eqref{lim} can be proved in a completely analogous way as in Proposition \ref{genextprop}. By Remark \ref{red}, one could also infer $\eqref{lim}$ as a corollary of \eqref{magg}. Finally, to see \eqref{magg}, one can simply follow the chain of logical equivalences of the previous part of the proof. This proves that (PE) is also necessary to the existence of the even extension.
\end{proof}

\subsection{Extension to Geometric Functionals}
Now consider an orthonormal basis of $\Lambda_m(\R^{m + n})$, denoted with $E_1,\dots, E_{\binom{m + n}{m}}$, where $$E_1 \doteq e_1\wedge\dots\wedge e_{m},$$ as done in \eqref{E1}. We define, for every $\tau \in \Lambda_m(\R^{m + n})$
\begin{equation}\label{psiG}
\Psi(\tau)\doteq \mathcal{G}(\langle \tau,E_2\rangle, \dots, \langle \tau,E_{\binom{m + n}{m}}\rangle, \langle \tau,E_1\rangle),
\end{equation}
and consequently the energy
\[
\Sigma_\Psi(T)\doteq \int_{E}\Psi(\vec T(x))\theta(x)d\mathcal{H}^m(x),
\]
for $T = \llbracket E, \vec T,\theta \rrbracket \in \mathcal{R}_m(\R^{n + m})$. For convenience, let us denote
\[
\phi(\tau) \doteq \left(\langle \tau,E_2\rangle, \dots, \langle \tau,E_{\binom{m + n}{m}}\rangle, \langle \tau,E_1\rangle\right).
\]
We have

\begin{prop}\label{Alm}
Let $\mathcal{G}$ be positively-1 homogeneous and convex, and define $\Psi$ as in \eqref{psiG}. Then, $\Sigma_\Psi$ fulfills Almgren's condition \eqref{ALM0}.
\end{prop}
\begin{proof}
Let $R,S \in \mathcal{R}_m(\R^{n + m})$, $\partial R = \partial S $,  $\spt S$ is contained in the vectorsubspace of $\R^n$ associated with a simple $m$ vector $\vec{S}_0$ of $\R^n$, and $\vec{S}(z)=\vec{S}_0$ for $\| S\|$-almost all $z$. Since $\partial R = \partial S$ we have that 
\[ \int \vec{R} \, d\|R\| = \int \vec{S} \, d\|S\| = \mathbb{M}(S) \,\vec{S}_0,\]
compare \cite[5.1.2]{FED}. Note that this implies by the linearity of $\phi$ that \[ \int \phi \circ \vec{R} \, d\|R\| = \int \phi \circ \vec{S} \, d\|S\| = \mathbb{M}(S) \,\phi \circ \vec{S}_0.\]
Now we may use Jensen inequality and the $1$-homogeneity of $\mathcal{G}$ to deduce that 
\begin{align*}
\int \Psi\circ \vec{R} d\|R\| &= 	\int \mathcal{G}\circ \phi \circ \vec{R} d\|R\|\ge \mathcal{G} \left( \int \phi \circ \vec{R} \, d\|R\| \right)  = \mathcal{G} \left( \mathbb{M}(S) \,\phi \circ \vec{S}_0\right)\\
 &= \mathbb{M}(S)\, \mathcal{G}\circ \phi \circ \vec{S}_0 = \int \Phi\circ \vec{S} d\|S\|\,,
\end{align*}
where we used again in the last line that $\mathcal{G}$ is 1-homogeneous. 
\end{proof}

\begin{remark}\label{posvarint}
If $\mathcal{G}$ is even, then $\Sigma_\Psi$ is a well-defined energy on varifolds. Notice that in this case, $\mathcal{G}$ is convex, even and 1-homogeneous. A simple computation in convex analysis shows that this imposes for $\mathcal{G}$ to be positive. This observation is what makes it impossible to extend an integrand $f$ as the one constructed in Section \ref{SCC} to an integrand defined on varifolds using the methods introduced here. 
\end{remark}

\begin{appendix}
\section{Currents, Varifolds and Geometric Functionals}\label{geom}

In this section we give the main definitions concerning currents and varifolds we have used throughout the paper. One can give more general definitions, namely flat, normal currents and general varifolds, see for instance \cite{FED,SIM}, but we limit ourselves to rectifiable currents and varifolds in order to keep the exposition as concise as possible.

\subsection{Multilinear algebra}\label{multa}

Let $n \in \N, m \ge 0$. We denote with $\Lambda_{m}(\R^{n + m})$ the space of $m$-vectors of $\R^{n + m}$, i.e. the vector space given by finite linear combinations of elements of the form
\[
v_{1}\wedge\dots\wedge v_m, \quad v_i \in \R^{n + m}, \forall 1\le i \le m.
\]
We also let $\Lambda_{m}^s(\R^{n + m})\subset \Lambda_{m}(\R^{n + m})$ be the space of non-zero simple $m$-vectors, i.e. all elements $\tau \in \Lambda_{m}(\R^{n + m})$ such that
\[
\tau = v_1\wedge\dots\wedge v_m
\]
for $v_1,\dots, v_m \in \R^{n + m}$. We define a \emph{canonical} basis of $\Lambda_{m}(\R^{n + m})$ as follows. Let $e_1,\dots, e_{n + m}$ be the vectors of the canonical basis of $\R^{n + m}$. Consider any multivector of length $m$, $I = (i_1,\dots,i_m)$, with $1\le i_1<\dots< i_m\le n+ m$. There are $\binom{n + m}{m}$ of these multivectors, and each one defines a simple $m$-vector
\[
E_j = e_{i_1}\wedge\dots\wedge e_{i_m}, \quad j \in \left\{1,\dots, \binom{n + m}{m}\right\}.
\]
It is easy to check that $\left\{E_1,\dots, E_{\binom{n + m}{m}}\right\}$ is a basis for $\Lambda_m(\R^{n + m})$. We also set
\begin{equation}\label{E1}
E_1\doteq e_1\wedge\dots\wedge e_m,
\end{equation}
while the ordering of the other indexes is arbitrary (but fixed).
\\
\\
The vector space $\Lambda_m(\R^{n + m})$ can be endowed with a scalar product that is defined on simple vectors as
\[
(v_1\wedge\dots\wedge v_m,w_1\wedge\dots\wedge w_m) \doteq \det(X),
\]
where $X \in \R^{m\times m}$ is defined as $X_{ij} = (v_i,w_j)$. We define a norm on $\Lambda_m(\R^{n + m})$ by setting $\|\tau\| = \sqrt{(\tau,\tau)}$. Analogously, one introduces the space of $m$-covectors of $\R^{n + m}$, $\Lambda^*_m(\R^{n + m})$, as the linear space generated by the wedge product of $m$ covectors of $\R^{n + m}$. An element $\eta \in \Lambda^*_m(\R^{n + m})$ acts by duality on elements of $\Lambda_m(\R^{n + m})$ in the following way. Let $\eta = \eta^1\wedge\dots\wedge \eta^m$ and $\tau = v_1\wedge\dots \wedge v_m$ (the general case follows by linearity). Then,
\[
\eta(\tau) = \det(Y),
\]
where $Y \in \R^{m\times m}$ is the matrix defined as $Y_{ij}\doteq \eta_i(v_j)$, $\forall 1 \le i,j \le m$.
\\
\\
With these definition at hand, we can consider the space $\mathcal{D}^m(\R^{n + m})$ as the space of smooth $m$-forms of $\R^{n + m}$ with compact support, namely
\[
\mathcal{D}^m(\R^{n + m}) \doteq C^\infty_c(\R^{n + m},\Lambda^*_m(\R^{n + m})).
\]
We endow $\Lambda_m^*(\R^{n + m})$ with a norm given by
\[
\|\eta\| \doteq \sup_{\tau \in \Lambda^s_m(\R^{n + m})\setminus\{0\}}\frac{|\langle\eta,\tau\rangle|}{\|\tau\|},
\]
hence we can consider on $\mathcal{D}^m(\R^{n + m})$ the norm
\[
\|\omega\|_\infty \doteq \sup_{x \in \R^{n + m}}\|\omega(x)\|.
\]

\subsection{Planes and rectifiable sets}

We denote with $\mathbb{G}(m,n+ m)$ the space of unoriented $m$-planes of $\R^{n + m}$. In \cite{DLDPKT}, we used the identification of $\mathbb{G}(m, n + m)$ with the space of orthogonal projections on $m$-planes
\[
\left\{P \in \R^{(m + n) \times (m + n)}: P =P^T, P^2 = P, \rank(P) = \tr(P) = m \right\}.
\]
It is not difficult to show that $\Lambda_m^s(\R^{n + m})$ can be identified with the space of \emph{oriented} $m$-dimensional planes of $\R^{m + n}$, see \cite[Section 2.1]{MGS1}. It is thus natural to introduce the two-to-one map
\begin{equation}\label{fiden}
f: \Lambda_m^s(\R^{n + m}) \to \mathbb{G}(m,n+ m)
\end{equation}
that takes $v_1\wedge\dots\wedge v_m \in \Lambda_m^s(\R^{n + m})$ to the projection on the $m$-plane spanned by $v_1,\dots, v_m$. Notice that $f$ is not injective since $f(\tau) = f(-\tau)$, $\forall \tau \in \Lambda^s_m(\R^{n + m})$.
\\
\\
We recall that a set $E \subset \R^{m + n}$ is called rectifiable of dimension $m$ if $$E = E_0\cup \bigcup_{j \ge 1}F_j(E_j),$$ where $\mathcal{H}^m(E_0) = 0$, $F_j \in C^1(\R^m,\R^{n + m})$, and $E_j \subset \R^{m + n}$ is Borel. To such a set $E$ it is possible to associate naturally a notion of approximate tangent plane, i.e. a map
\[
x\mapsto T_xE \in \mathbb{G}(m,n+ m).
\]
For the definition of $T_xE$, we refer the reader to \cite[Section 3.1]{SIM}. An \emph{orientation} $x\mapsto \vec T(x)$ of $T_xE$ is a Borel map $\vec T \in L^\infty(E,\Lambda_m^s(\R^{n + m}))$ with $\|\vec T(x)\| = 1$ for $\mathcal{H}^m\llcorner E$ a.e. $x \in \R^{n + m}$, and $\vec T(x) \in f^{-1}(T_xE)$, for $\mathcal{H}^m\llcorner E$ a.e. $x \in \R^{n + m}$, where $f$ is the map defined in \eqref{fiden}.

\subsection{Varifolds}

A $m$-dimensional rectifiable varifold $V$ is a measure on $\R^{n + m}\times \mathbb{G}(n + m,m)$ given by
\begin{equation}\label{VAR}
V(g) \doteq \int_{E}g(x,T_xE)\theta(x)d\mathcal{H}^m(x), \quad \forall g \in \R^{n + m}\times \mathbb{G}(n + m,m)
\end{equation}
where $E$ is $m$ dimensional rectifiable set and $\theta \in L^1(E;\mathcal{H}^m\llcorner E)$. The varifold is called \emph{integer rectifiable} if in addition $\theta$ has values in $\mathbb{N}\setminus\{0\}$. The notation for the varifold $V$ defined as in \eqref{VAR} is
\[
V =\llbracket E,\theta \rrbracket.
\]

\subsection{Currents}\label{currents}

A rectifiable current of dimension $m$, denoted by $T$, is a linear functional over $\mathcal{D}^m(\R^{n + m})$ represented as:
\[
T(\omega) \doteq \int_E\langle\omega(x),\vec T(x)\rangle\theta(x)d\mathcal{H}^m(x),
\]
where $E$ is an $m$-rectifiable subset of $\R^{n + m}$, $\vec T(x)$ is an orientation of $T_xE$, and $\theta \in L^1(E;\mathcal{H}^m\llcorner E)$. Such a current $T$ is denoted as 
\[
T =\llbracket E,\vec T,\theta \rrbracket.
\]
The mass of the current $T$ is defined as
\[
\M(T) \doteq \int_E|\theta(x)|d\mathcal{H}^m(x),
\] 
and we introduce the notion of boundary $\partial T$ of a rectifiable current $T$ as the $m - 1$ dimensional current
\[
\partial T( \omega) \doteq T(d\omega).
\]
We restrict our attention to the space of \emph{integer rectifiable currents} of dimension $m$, $\mathcal{R}_m(\R^{n + m})$, defined as the space of $m$-dimensional rectifiable currents $T$ with finite mass and for which $\theta$ has values in $\mathbb{N}\setminus \{0\}$.
\\
\\
Given $T = \llbracket E,\vec T,\theta\rrbracket$ and an injective vector-field $X \in C^1(\R^{n + m},\R^{n + m})$, we define the pushforward of $T$ as the current $X_{\#}(T) \in \mathcal{R}_m(\R^{n + m})$ defined by
\[
X_{\#}(T) \doteq \llbracket X(E), \vec\xi, \theta\circ X^{-1} \rrbracket,
\]
where
\begin{equation}\label{orienting}
\vec\xi(X(x))\doteq \frac{DX(x)v_1(x)\wedge\dots\wedge DX(x)v_m(x)}{\|DX(x)v_1(x)\wedge\dots\wedge DX(x)v_m(x)\|},\quad \forall x \in E,
\end{equation}
if $\vec T(x) = v_1(x)\wedge\dots\wedge v_m(x)$. Analogously, for a varifold $V =\llbracket E,\theta \rrbracket$,
\[
X_\#V \doteq \llbracket X(E), \theta\circ X^{-1}\rrbracket.
\]
\\
\\
Notice that to every current $T \in \mathcal{R}_m(\R^{n + m})$ one can associate an integer rectifiable varifold $V_T$ in the obvious way
\[
V_T(g) \doteq \int_{E}g(x,f(\vec\tau(x)))\theta(x)d\mathcal{H}^m(x),
\]
where $f$ is the map defined in \eqref{fiden}.
\\
\\
Let us explain how to give to a graph of a Lipschitz map a structure of current, hence also of varifold. We essentially follow the theory developed in \cite{MGS1,MGS2}. We also refer the reader to \cite{DLDPKT}, where this discussion was made for giving the graph a structure of varifold. Let $\Omega\subset \R^m$ be open and bounded and let $u \in \Lip(\Omega,\R^n)$. Then, the graph of $u$ defined as
\[
\Gamma_u = \{(x,u(x)):x\in \Omega\}
\]
is $m$-rectifiable. Furthermore, as proved in \cite[Sec. 1.5, Th. 5]{MGS1} its approximate tangent plane is, at a.e. $x_0 \in \Omega$, given by the orthogonal projection on
\[
\pi(x_0)\doteq \spn\{\partial_1(x,u(x))|_{x = x_0},\dots, \partial_m(x,u(x))|_{x = x_0}\} = \spn\{(f_1,\partial_1u(x_0))^T,\dots, (f_m,\partial_m u(x_0))^T\},
\]  
where $f_1,\dots,f_m$ are the elements of the canonical basis of $\R^m$. Define $v_i(x)\doteq (f_i,\partial u_i(x))^T$. The orientation we define on $\pi(x_0)$ is the natural one:
\begin{equation}\label{orientinggraph}
\vec\xi_u(x)\doteq \frac{v_1(x)\wedge\dots\wedge v_m(x)}{\|v_1(x)\wedge\dots\wedge v_m(x)\|}.
\end{equation}
Given a Borel function $\theta \in L^1(\Gamma_u;\mathcal{H}^m\llcorner \Gamma_u)$, we define the current $T_{u,\theta} = \llbracket \Gamma_u,\vec\xi_u,\theta \rrbracket$ and $\beta(x) \doteq \theta(x,u(x))$. Through the area formula, see for instance \cite[Proposition 6.4]{DLDPKT}, we have
\[
\M(T_{u,\theta}) = \int_{\Omega}\mathcal{A}(Du(x))\beta(x)dx,
\]
where
\begin{equation}\label{e:area}
\mathcal{A} (X) \doteq \sqrt{\det ({\rm id}_{\mathbb R^{m\times m}} + X^T X)}
\end{equation}
is the \emph{area function}. Notice that in the case $n = m = 2$, $\mathcal{A}$ has the form \eqref{area}. In particular, by the definition of norm of a $m$-vector, we notice that
\begin{equation}\label{orientinggraph2}
\mathcal{A}(Du) = \|v_1\wedge\dots \wedge v_m\|,
\end{equation}
where we have used the notation of \eqref{orientinggraph}.

\subsection{Geometric Functionals}\label{GM}

Given a smooth and $1$-homogeneous function $\Psi: \Lambda_s(\R^{n + m}) \to \R$, we can define the functional on rectifiable currents $T = \llbracket E, \vec T, \theta \rrbracket$
\begin{equation}\label{ENCUR}
\Sigma_\Psi(T)\doteq \int_{E}\Psi(\vec T(x))\theta(x)d\mathcal{H}^m(x)= \int \Psi(\vec{T}) \, d\|T\|,
\end{equation}
as done in $\eqref{EN}$. On varifolds, given a smooth integrand $F: \mathbb{G}(m,n+m)\to \R$, the counterpart of the previous energy has the form
\begin{equation}\label{ENVAR}
\Sigma'_F(V)\doteq \int_{E}F(T_xE)\theta(x)d\mathcal{H}^m(x),
\end{equation}
if $V = \llbracket E, \theta\rrbracket$. In particular, any even integrand $\Psi:\Lambda_s(\R^{n + m}) \to \R$ as above allows us to define a functional on varifold too. The minimal hypotheses that one requires on an integrand $\Psi$ to get lower semicontinuity of the energy $\Sigma_\Psi$, see \cite[Section 5.1]{FED}, is that $\Psi$ has positive values and it satisfies \emph{Almgren's ellipticity condition}, i.e.
\begin{equation}\label{ALM0}
\Sigma_\Psi(T)-\Sigma_\Psi(Q) \ge 0,
\end{equation}
whenever $\partial T = \partial Q$, $Q$ has support contained in an $m$ subspace of $\R^{n + m}$ whose orienting $m$-vector is $\vec\tau = v_1\wedge\dots\wedge v_m$ and the orientation of $Q$ is given by $\tau$. We say it satisfies a uniform Almgren ellipticity condition if there exists $\eps > 0$ such that
\begin{equation}\label{UALM}
\Sigma_\Psi(T)-\Sigma_\Psi(Q) \ge \eps(\M(T) - \M(Q)).
\end{equation}
We give now the definition of stationarity in the sense of currents (or varifolds). Fix an energy $\Sigma_\Psi$ and let $\mathcal{U} \subset \R^{m +n }$ be open. Given any function $g \in C^\infty_c(\mathcal{U},\R^{n + m})$, we define the flow $X_\eps(x) \doteq \gamma_x(\eps)$, where $\gamma_x$ is the solution of the ODE
\begin{equation}\label{flowg}
\begin{cases}
\gamma'(t) = g(\gamma(t))\\
\gamma(0) = x.
\end{cases}
\end{equation}
We define the variation of $T$ with respect to the vector field $g \in C^1_c(\mathcal{U}; \R^{m + n})$ as
\begin{equation}\label{delta}
[\delta_\Psi T] (g) \doteq \lim_{\varepsilon \to 0}\frac{\Sigma_\Psi((X_\varepsilon)_\#T) - \Sigma_\Psi(T)}{\varepsilon}.
\end{equation}

Finally, the current $T$ is said to be \emph{stationary in $\mathcal{U}$} if $[\delta_\Psi T] (g)=0, \forall g \in C^1_c(\mathcal{U}; \R^{m + n})$. With obvious modifications, this definition holds for varifolds as well.
\\
\\
A simple computation shows the following characterization of the first variation of a geometric functional

\begin{lemma}\label{expvar}
Let $T = \llbracket T,\vec\tau,\theta \rrbracket$ with $\vec \tau = \tau_1\wedge\dots\wedge \tau_m$ and $\|\vec\tau(x)\| = 1$. For any, $g \in C_c^1(\mathcal{U},\R^{m + n})$,
\begin{align*}
[\delta_\Psi T](g) = \sum_{i = 1}^m\int_{E}\langle d\Psi(\vec\tau(x)),\tau_1(x)\wedge\dots\wedge Dg(x)\tau_i(x)\wedge\dots\wedge \tau_m(x)\rangle\theta(x)d\mathcal{H}^m(x).
\end{align*}
\end{lemma}

\subsection{Functionals on graphs}
It was shown in \cite[Section 6]{DLDPKT} that from a functional defined on varifolds, one can define a functional on graphs, simply using the area formula. To do so, we introduced the map $h:\R^{n\times m} \to \R^{(n + m)\times (n + m)}$ defined as
\[
h(X)\doteq M(X)S(X)M(X)^T
\]
where
\begin{equation}\label{MS}
M(X)\doteq
\left(
\begin{array}{c}
\id_m\\
X
\end{array}
\right)\quad \text{ and }\quad S(X)\doteq(M(X)^TM(X))^{-1},
\end{equation}
or, more explicitely,
\begin{equation}\label{hform}
h(X) =  \left[\begin{array}{c|c}
S(X)&S(X)X^T\\ \hline
XS(X)&XS(X)X^T
\end{array}\right].
\end{equation}
This represents the orthogonal projection on the plane $$\tau(X) = \spn\{M(X)^Te_1,\dots, M(X)^Te_{n + m}\},$$ and is the parametrization of one chart of $\mathbb{G}(m,n+m)$. If $F$ is an integrand as in $\eqref{ENVAR}$, we can define
\begin{equation}\label{fF}
f(X) \doteq F(h(X))\mathcal{A}(X),
\end{equation}
where $\mathcal{A}$ is the area element defined in \eqref{e:area}. The following holds
\begin{equation}\label{eqvar}
\Sigma'_F(V_{u,\theta}) = \int_{\Omega}f(Du(x))\beta(x)dx,
\end{equation}
for every $V_{u,\theta} = \llbracket\Gamma_u,\theta\rrbracket$, where $\beta(x) = \theta(x,u(x))$. In \cite[Proposition 6.6]{DLDPKT}, we proved the previous equality in the case $\theta \equiv 1$, but the case with multiplicity holds with the same proof. One can do the same for functionals defined on currents in the following way. Let $\Psi \in C^\infty(\Lambda_m^s(\R^{n + m}))$ and associate, to $X \in \R^{n\times m}$, the simple vector
\[
W(X) \doteq M^1(X)\wedge\dots\wedge M^m(X),
\]
where $M^i(X)$ denotes the $i$-th column of the matrix $M(X)$ defined in \eqref{MS}. If we define
\begin{equation}\label{fpsi}
f(X) \doteq \Psi(W(X))\mathcal{A}(X),
\end{equation}
the area formula once again yields the equality
\begin{equation}\label{eqcurr}
\Sigma_\Psi(T_{u,\theta}) = \int_{\Omega}f(Du(x))\beta(x)dx,
\end{equation}
for every $T_{u,\theta} = \llbracket\Gamma_u,\vec\xi_u,\theta\rrbracket$, where $\beta(x) = \theta(x,u(x))$.

Finally, let us discuss the link between stationarity for geometric objects and stationarity in the graph sense. We refer the interested reader to \cite[Proposition 6.8]{DLDPKT} for a more precise statement in the case of multiplicity $1$ graphs.

\begin{prop}\label{vargen}
Let $F: \mathbb{G}(m,n+m) \to \R$ or $\Psi: \Lambda_m^s(\R^{n + m})\to \R$ be given and define $f$ through formula \eqref{fF} or \eqref{fpsi}, respectively. Let $\Omega$ be a Lipschitz, bounded, open subset of $\R^m$ and $\beta \in L^1(\Omega,\R^+)$. A map $u\in\Lip(\Omega,\mathbb{R}^n)$ satisfies
\begin{equation*}
\begin{cases}
\displaystyle \int_{\Omega}\langle Df(D u),D v\rangle\beta(x) dx = 0&\forall v\in C^1_c(\Omega,\mathbb{R}^n) \vspace{1mm}\\ 
\displaystyle \int_{\Omega}\langle Df(D u), D u D \phi(x)\rangle\beta(x) dx - \int_{\Omega}f(D u)\dv(\phi)\beta(x) dx = 0\;\; &\forall \phi\in C_c^1(\Omega,\mathbb{R}^m),
\end{cases}
\end{equation*}
if and only if the rectifiable varifold $V_{u,\theta} = \llbracket \Gamma_u,\theta \rrbracket$ or the rectifiable current $T_{u,\theta} = \llbracket \Gamma_u,\vec\xi_u,\theta \rrbracket$, where $\theta(x,y) = \beta(x), \forall (x,y) \in \R^{m + n}$, are stationary with respect to $\Sigma'_F$ or $\Sigma_\Psi$, respectively.
\end{prop}
\begin{proof}
Since the proof is essentially the same of \cite[Proposition 6.8]{DLDPKT}, we only sketch it. In \cite[Proposition 6.8]{DLDPKT}, only the varifold case was considered, hence let us consider the case of functionals defined on currents here.
\\
\\
\indent\emph{Step 1:} Reduction to special vector fields.
\\
\\
Define, for any $g\in C_c^1(\Omega\times\mathbb{R}^n,\mathbb{R}^{m + n})$, $g = (g_1,\dots,g_{n + m},)$, two fields $g^1 \doteq (g_1,\dots,g_m,0,\dots,0)$ and $g^2 \doteq (0,0,g_{m + 1},\dots, g_{n + m})$, so that $g = g^1 + g^2$. From now on, consider $g$ fixed. From Lemma \ref{expvar}, we see that the first variation $[\delta_\Psi T]$ (see the notation introduced in \eqref{delta}) enjoys the following properties:
\begin{equation}\label{linvar}
[\delta_\Psi T_u](g) = [\delta_\Psi T_u](g^1) + [\delta_\Psi T_u](g^2). 
\end{equation}
and
\begin{equation}\label{locality}
[\delta_\Psi T_u](g) = [\delta_\Psi T_u](h),\qquad \forall g,h \in  C_c^1(\Omega\times\mathbb{R}^n,\mathbb{R}^{m + n}), g|_{\spt(\Gamma_u)} = h|_{\spt(\Gamma_u)}\, .
\end{equation}
\eqref{linvar} is trivial, while to show \eqref{locality}, simply notice that if
\[
g|_{\spt(\Gamma_u)} = h|_{\spt(\Gamma_u)},
\]
then 
\begin{equation}\label{tangent}
Dg(x,u(x))w = Dh(x,u(x))w, \text{ for a.e. } x \in \Omega \text{ and } w \in T_{(x,u(x))}\Gamma_u.
\end{equation}
By exploiting the explicit form of the first variation written in Lemma \ref{expvar}, \eqref{locality} follows at once. From \eqref{locality} we conclude that it suffices to consider the first variation of the current $T_u$ for vector fields $g$ of the form
\begin{equation}\label{g}
g(x,y) = \chi(y)G(x,u(x)),
\end{equation}
for $G \in C_c^1(\Omega\times\mathbb{R}^n,\mathbb{R}^{n + m})$, and $\chi \in C^\infty_c(\R^n)$, $\chi(y) \equiv 1$ on $B_{2M}(0) \subset \R^n$ and $\chi(y)\equiv 0$ on $B_{3M + 1}(0)$, where  $M \doteq \max_{x \in \pi(\spt(G))}\|u(x)\|$. $\pi: \R^{m + n} \to \R^m$ here denotes the projection $\pi(x,y) \doteq x, \forall (x,y) \in \R^{m + n}, x \in \R^m,y\in \R^n$. 
\\
\\
\indent\emph{Step 2:} Inner variations.
\\
\\
We let $X_\eps$ be the flow generated by $g^1$, for $g$ as in \eqref{g}. It is easy to see that
\[
X_\eps(x,u(x)) = (Z_\varepsilon(x),u(x)), \forall x \in \Omega,
\]
where $Z_\varepsilon$ is the flow generated by the field $x\mapsto G^1(x,u(x))$. Using this information, one readily checks that
\[
(X_\varepsilon)_\#T_u = \llbracket X_\varepsilon(\Gamma_u),\vec\tau_\eps,\theta\circ (X_\eps)^{-1}\rrbracket = \llbracket \Gamma_{u\circ Z_{-\varepsilon}(\cdot)}, \vec\xi_{u \circ Z_{-\eps}}, \theta\circ(\cdot, Z_{-\eps}(\cdot))\rrbracket = T_{u\circ Z_{-\eps},\theta\circ (\cdot,Z_{-\eps}(\cdot))}.
\]
Through formula \eqref{eqcurr}, we see that
\begin{align*}
\Sigma_\Psi(T_{u\circ Z_{-\eps},\theta\circ (\cdot,Z_{-\eps}(\cdot))}) & = \int_\Omega f(D(u\circ Z_{-\eps})) \beta\circ Z_{-\eps}(x) \, dx\,\\
&= \int_\Omega f(Du(x)DZ_{-\eps}(Z_\eps(x))) \, \beta(x) \det(DZ_\eps) \, dx \,. 
\end{align*}
By taking the derivative at $\eps = 0$ of the previous expression, we get
\begin{equation}\label{eqinner}
[\delta_\Psi T](g^1) =  \int_{\Omega}\langle Df(Du), Du D(G^1(x,u(x)))\rangle\beta(x) dx - \int_{\Omega}f(D u)\dv(G^1(x,u(x)))\beta(x) dx
\end{equation}
\\
\\
\indent\emph{Step 3:} Outer variations.
\\
\\
Similarly to the case above, consider the flow $Y_\eps$ generated by $g^2$, for $g$ as in \eqref{g}. Then, one checks that
\[
Y_\varepsilon(x,u(x)) = (x, u(x) + \varepsilon G^2(x,u(x)),
\]
and hence
\[
(Y_\varepsilon)_\#T_u = \llbracket Y_\varepsilon(\Gamma_u),\vec\tau_\eps,\theta\circ Y_{-\eps}\rrbracket = \llbracket \Gamma_{ v_{\varepsilon}},\vec\xi_{v_\eps}, \theta\circ Y_{-\eps}\rrbracket =T_{v_\eps,\theta\circ Y_{-\eps}}.
\]
By \eqref{eqcurr}, we write
\[
\Sigma((Y_\eps)_\#T_u)= \int_\Omega f(Du+\eps D (G^2(x,u(x)))) \beta \, dx,
\]
whose derivative at $\eps = 0$ yields
\begin{equation}\label{eqouter}
[\delta_\Psi T](g^2)  =\int_{\Omega}\langle Df(D u), D (G^2(x,u(x)))\rangle\beta(x) dx.
\end{equation}
\\
\\
Now the Proposition follows at once from \eqref{linvar}-\eqref{locality}-\eqref{eqinner}-\eqref{eqouter}.

\end{proof}

\end{appendix}
\bibliographystyle{plain}
\bibliography{Last}
\end{document}